\documentclass[oneside,11pt,reqno]{amsart}
 
\usepackage{amsmath, amscd}

\usepackage{amsfonts}
\usepackage{amssymb}
\usepackage{amsxtra}
\usepackage{tikz-cd}
\usepackage{hyperref}
\usepackage{enumerate}
\usepackage{amsthm}

\usepackage{amsmath}
\usepackage{amssymb}
\usepackage{graphicx, overpic}
\usepackage{caption}
\usepackage{subcaption}
\usepackage{pinlabel}
\usepackage{floatrow}
\usepackage{soul}
\usepackage{tikz}
\usetikzlibrary{decorations.pathreplacing}
\usetikzlibrary{decorations.markings}
\usetikzlibrary{shapes.geometric}
\usepackage{color}
\usetikzlibrary{shapes.gates.logic.US,trees,positioning,arrows}
\usetikzlibrary{knots}
 \usepackage{float}
 
 \usepackage{eucal}
\usepackage[all]{xy}
\usepackage{multicol}

\usepackage{mathtools}

\newtheorem{thm}{Theorem}[section]
\newtheorem{prop}[thm]{Proposition}
\newtheorem{lem}[thm]{Lemma}

\newtheorem*{main}{Main Theorem}

\theoremstyle{definition}
\newtheorem{defn}[thm]{Definition}
\newtheorem{rem}[thm]{Remark}
\newtheorem{conv}[thm]{Convention}

\newtheorem{ques}[thm]{Question}

\newtheorem*{prop-MC}{Proposition 4.28}

\renewcommand{\bar}[1]{\overline{#1}}

\newcommand{\set}[2]{\{\,{#1} \mid {#2} \,\}}

\newcommand{\field}[1]{\mathbb{#1}}
\newcommand{\Z}{\field{Z}}

\newcommand{\R}{\field{R}}

\newcommand{\E}{\field{E}}

\newcommand{\F}{\field{F}}

\newcommand{\NN}{\field{N}}

\newcommand{\hct}[1]{\textcolor{blue}{#1}}

\DeclareMathOperator{\Dist}{Dist}

\DeclareMathOperator{\CAT}{CAT}

\DeclareMathOperator{\Area}{Area}

\DeclareMathOperator{\diam}{diam}

\newcommand{\showcomments}{yes}

\usepackage{ifthen}
\newsavebox{\commentbox}
{\ifthenelse{\equal{\showcomments}{yes}}%
{\footnotemark
    \begin{lrbox}{\commentbox}
    \begin{minipage}[t]{1.25in}\raggedright\sffamily\tiny
    \footnotemark[\arabic{footnote}]}
{\begin{lrbox}{\commentbox}}}%
{\ifthenelse{\equal{\showcomments}{yes}}%
{\end{minipage}\end{lrbox}\marginpar{\usebox{\commentbox}}}
{\end{lrbox}}}

\begin{document}

\title[Superexponential Dehn functions inside $\CAT(0)$ groups]{Superexponential Dehn functions inside $\CAT(0)$ groups}

\begin{abstract}
 We construct 4--dimensional $\CAT(0)$ groups containing finitely presented subgroups whose Dehn functions are $\exp^{(n)}(x^m)$ for integers $n, m \geq 1$ and 6--dimensional $\CAT(0)$ groups containing finitely presented subgroups whose Dehn functions are $\exp^{(n)}(x^\alpha)$ for integers $n \geq 1$ and $\alpha$ dense in $[1,\infty)$. This significantly expands the known geometric behavior of subgroups of $\CAT(0)$ groups.
\end{abstract} 

\author{Noel Brady}
\address{University of Oklahoma, Norman, OK 73019-3103, USA}
\email{nbrady@ou.edu}

\author{Hung Cong Tran}
\address{University of Oklahoma, Norman, OK 73019-3103, USA}
\email{Hung.C.Tran-1@ou.edu}

\maketitle

\section{Introduction}

$\CAT(0)$ geometry plays a fundamental role in geometric group theory and low dimensional topology. For example, in the proof of the Virtual Fibering Theorem \cite{MR3104553}, $\CAT(0)$ cube complexes are used to encode immersed incompressible surfaces in a 3--manifold and this $\CAT(0)$ cubical structure is used to understand finite covers of the manifold. In another direction, the $\CAT(0)$ framework and Morse theory perspective introduced in \cite{MR1465330} to produce groups of type FP$_2$ which are not finitely presented has resulted in a range of applications including \cite{MR1943744}, \cite{MR4176549} and \cite{2018arXiv180910594K}.




$\CAT(0)$ groups are fundamental groups of compact, path metric spaces whose universal covers satisfy the $\CAT(0)$ triangle inequality. See \cite{MR1744486} for definitions and background on $\CAT(0)$ spaces and groups. $\CAT(0)$ groups have strongly controlled topology and algebra, including having finite Eilenberg-Maclane spaces, solvable word and conjugacy problems, tightly controlled solvable subgroups, and quadratic upper bounds for their Dehn functions. They have been intensively studied over the years. 

Much less is known about the geometry and structure of finitely presented subgroups of $\CAT(0)$ groups. In this paper we focus on Dehn functions. In \cite{MR1489138} there are examples of automatic groups containing finitely presented subgroups with exponential or polynomial Dehn functions of arbitrary degree. The ambient groups in \cite{MR1489138} are $\CAT(0)$. In \cite{MR3705143} there are examples of $\CAT(0)$ groups containing finitely presented subgroups whose Dehn functions are $x^\alpha$ for a dense set of exponents $\alpha \in [2,\infty)$. 
Throughout this time there were no examples of $\CAT(0)$ groups containing finitely presented subgroups with superexponential Dehn functions. Question 10.7 of \cite{Br07} asks if there is a universal upper bound on the Dehn functions of finitely presented subgroups of $\CAT(0)$ groups. 


The goal of the present paper is to shed light on the question above. Such a universal upper bound would have to be greater than any finite iteration of exponential functions. 

\begin{main}[paraphrased]
Let $n$ be a positive integer and define $\exp^{(n)}(x)$ inductively by $\exp^{(1)}(x) = \exp(x)$ and $\exp^{(k+1)}(x) = \exp(\exp^{(k)}(x))$ for $k \geq 1$. Then 
\begin{enumerate}
    \item There are $4$--dimensional $\CAT(0)$ groups containing finitely presented subgroups whose Dehn functions are equivalent to $\exp^{(n)}(x^m)$ for integers $m \geq 1$. 
    \item  There are 6--dimensional $\CAT(0)$ groups containing finitely presented subgroups whose Dehn functions are equivalent to $\exp^{(n)}(x^\alpha)$ for $\alpha$ dense in $[1,\infty)$.
\end{enumerate}
\end{main}

The $\CAT(0)$ groups in the Main Theorem are obtained from two constructions of non-positively curved spaces; the \emph{ultra-convex chaining} procedure and the \emph{factor-diagonal chaining} procedure. These are described in detail in Section~5. The ultra-convex chaining procedure builds a non-positively curved (that is, locally $\CAT(0)$) 2--complex from a sequence of non-positively curved 2--complexes by identifying a non-convex rose (bouquet of circles) in one 2--complex in the sequence with an ultra-convex rose in the next 2--complex. Theorem~\ref{prop:fbyf} provides a key ingredient for this procedure; namely, a $\CAT(0)$ $F_\ell \rtimes F_k$ group with $F_k$ ultra-convex. 
The factor-diagonal chaining process builds non-positively curved spaces from a sequence of metric product spaces by identifying a factor space in one product with a diagonal subspace in the next.

The subgroups with large Dehn functions are constructed using a standard technique; namely, by doubling groups over highly distorted subgroups. The subgroup distortion functions contain iterated exponentials. These are the result of chaining together a sequence of the hyperbolic free-by-free groups constructed in Theorem~\ref{prop:fbyf}. 

The tricky part is to embed these groups into $\CAT(0)$ groups. This is achieved by the factor-diagonal chaining construction and a group embedding result. The group embedding part of the main theorem is provided by Proposition~\ref{prop:embed0} which uses a result of Bass to prove that a particular morphism of graphs of groups induces an inclusion of fundamental groups. The structure of the group embedding proposition requires that the factor-diagonal chaining parallels (and has as many terms as) the ultra-convex chaining. 

The various power functions inside the iterated exponential are obtained by choosing suitable $\CAT(0)$ ingredients to start the construction. For conclusion (1) of the Main Theorem we use $\CAT(0)$ groups of the form $(F_k \rtimes \Z) \times \Z$ and for conclusion (2) we use the ambient $\CAT(0)$ groups which contain the snowflake groups of \cite{MR3705143}. 

This paper is organized as follows. Section 2 provides background on Dehn functions, distortion functions of subgroups, and $\CAT(0)$ spaces. Section~3 establishes the main group embedding result. Section~4 concerns the computation of the Dehn functions of the finitely presented subgroups. In Section~5 the ambient $\CAT(0)$ groups are constructed. Section~6 contains a statement and proof of the Main Theorem and a list of open questions which are related to this research.

\subsection*{Acknowledgements} We thank Eduardo Mart\'{i}nez-Pedroza for helpful conversations about the questions in Section~6. We thank the referee for their careful reading and commentary and for providing some new references.

The first author was supported by Simons Foundation collaboration grant \#430097. 
The second author was supported by an AMS-Simons Travel Grant.

\section{Background}

The main theorem of this paper concerns Dehn functions of finitely presented subgroups of $\CAT(0)$ groups. The large Dehn functions arise from taking doubles of groups over highly distorted subgroups. The first subsection provides background on equivalence of functions, Dehn functions of groups, and distortion functions of subgroups. The second subsection gives background on non-positively curved, piecewise euclidean 2--complexes and on constructions related to products of non-positively curved spaces. 

\subsection{Dehn functions and subgroup distortion} 

The Dehn functions and distortion functions defined below are all computed up to an equivalence which we describe. 

\begin{conv}
\label{cv}
Let $\mathcal{M}$ be the collection of all functions from $[1,\infty)$ to $[1,\infty)$. Let $f$ and $g$ be arbitrary elements of $\mathcal{M}$. We say that $f$ is \emph{dominated} by $g$, denoted $f \preceq g$, if there is a constant $C\geq 1$ such that $f(x) \leq Cg(Cx) + Cx$ for all $x\geq 1$. We say that $f$ is \emph{equivalent} to $g$, denoted $f\sim g$, if $f\preceq g$ and $g\preceq f$. This defines an equivalence relation on $\mathcal{M}$. \end{conv}

We now recall the definitions of Dehn function and subgroup distortion.

\begin{defn}[Dehn function and isoperimetric inequality]
Let $G=\langle S|R \rangle$ be a finitely presented group and $w$ a word in the generating set $S$ representing the trivial element of $G$. We define the \emph{area} of $w$ to be

$$\Area(w) = \min \set{N\in\NN}{\exists \text{ equality } w=\prod_{i=1}^{N} u_ir_iu_i^{-1} \text{ freely, where } r_i\in R^\pm}$$
The \emph{Dehn function} $\delta(x)$ of the finite presentation $G=\langle S|R \rangle$ is given by
$$\delta(x)=\max\set{\Area(w)}{w\in \ker(F(S)\to G), |w|\leq x}$$
where $|w|$ denotes the length of the word $w$.

Using Convention~\ref{cv} it is not difficult to show that two finite presentations of the same group define equivalent Dehn functions, we therefore speak of ``the'' Dehn function of $G$, which is
well-defined up to equivalence.

If the Dehn function of the finitely presented group $G$ is dominated by a polynomial function we say that $G$ \emph{satisfies a polynomial isoperimetric inequality.}
\end{defn}

\begin{defn}[subgroup distortion]
Let $H \leq G$ be a pair of finitely generated groups, and let $d_H$ and $d_G$ be the word metrics associated to a choice of finite generating set for each. The \emph{distortion} of $H$ in $G$ is the function
$$\Dist^G_H(x) = \max\set{d_H(1,h)}{h\in H \text{ with } d_G(1,h)\leq x}.$$
Up to the equivalence relation in Convention~\ref{cv}, this function is independent of the choice of word metrics $d_G$ and $d_H$.
\end{defn}

The following theorem relates subgroup distortion to the Dehn functions of doubles.

\begin{thm}[\cite{MR1744486}, III.$\Gamma$.6.20]
\label{BH1}
Let $G$ be a finitely presented group with Dehn function $\delta_G$ and let $H\leq G$ be a finitely presented subgroup. Let $\Dist^G_H$ be the distortion of $H$ in $G$ with respect to some choice of word metrics. The Dehn function $\delta_\Gamma$ of the double group $\Gamma=G*_H G$ satisfies
$$\Dist^G_H(x)\preceq \delta_\Gamma(x)\preceq x\bigl(\delta_G\circ \Dist^G_H\bigr)(x).$$
\end{thm}


\subsection{{${\rm CAT}(0)$} geometry.}
The reader is referred to \cite{MR1744486} for definitions of and background on $\CAT(0)$ spaces and groups. The phrase \emph{non-positively curved} means the same thing as \emph{locally $\CAT(0)$}. The universal covering space of a non-positively curved space is a $\CAT(0)$ space. A $\CAT(0)$ group is one which acts properly, cocompactly by isometries on a $\CAT(0)$ metric space. In particular, the fundamental group of a compact, non-positively curved space is a $\CAT(0)$ group.

This paper makes extensive use of non-positively curved piecewise euclidean 2--complexes and metric products of these complexes. We include pertinent definitions and results. 

A \emph{convex polygon} in $\E^2$ is a compact subspace of $\E^2$ which is a finite intersection of halfspaces. 

\begin{defn}[PE 2--complex]
A \emph{piecewise euclidean 2--complex} $X$ is a 2--dimensional cell complex with the following additional structure. 
\begin{enumerate}
    \item The 0--skeleton $X^{(0)}$ is a discrete set.
    \item Each 1--cell $e$ of $X$ consists of a segment $C_e \subseteq \R \subseteq \R^2 = \E^2$ and an attaching map $\varphi_e: \partial C_e \to X^{(0)}$. An \emph{admissible characteristic map} of a $1$--cell $e$ is the standard characteristic map $\chi_e: C_e \to X^{(1)}$ precomposed with an isometry of $\E^2$. 
    \item Each 2--cell $f$ of $X$ consists of a convex polygon $C_f \subseteq \E^2$ and an attaching map $\varphi_f: \partial C_f \to X^{(1)}$ satisfying the following glueing property. The restriction of $\varphi_f$ to each edge of $C_f$ is an admissible characteristic map of a 1-cell of $X$. 
\end{enumerate}
\end{defn}

\begin{defn}[link]
Assume that the length of the shortest $1$--cell of $X$ is 1. Given $v \in X^{(0)}$, one can define the \emph{link of $v$ in $X$} to be 
$$
Lk(v, X) \; =\;\partial B_{\frac{1}{4}}(v,X). 
$$
This is the boundary of a metric ball of radius 1/4 centered on $v$. 
A piecewise spherical metric on $Lk(v,X)$ is defined by taking the induced path metric on $Lk(v,X)$ defined above and rescaling by 4. In this rescaled metric, the length of an edge of $Lk(v,X)$ is equal to the angle at the vertex $w$ in the polygon $C_f$ where $\chi_f(w) = v$ and $\chi_f(C_f \cap B_{\frac{1}{4}}(w, C_f))$
is the given edge. 
\end{defn}

The \emph{link condition} of Theorem~\ref{lkcondition} below will be used multiple times in Section~5 of this paper. 
See \cite{MR1744486}, II.5.5 and II.5.6 for a proof. 

\begin{thm}[link condition for non-positively curved 2--complexes]
\label{lkcondition}
A finite piecewise euclidean 2--dimensional complex $K$ is a non-positively curved space if and only if for each vertex $v\in K$ every injective loop in the piecewise spherical metric on $Lk(v, K)$ has length at least $2\pi$.
\end{thm}


\medskip

The factor-diagonal chaining construction in Section~5 of this paper uses properties of metric products of non-positively curved spaces. Here are relevant definitions and background results.

Let $(X, d_1)$ and $(Y, d_2)$ be two metric spaces. 
The \emph{product metric} on $X \times Y$ is defined by 
$$
d((x_1, y_1), (x_2, y_2)) \; =\; \sqrt{(d_1(x_1, x_2))^2 + (d_2(y_1, y_2))^2}.
$$

Let $A$ and $X$ be non-positively curved metric spaces. A map $f\!\!: A \to X $ is said to be a \emph{locally isometric embedding} if every point in $A$ has a neighborhood $N$ such that $f_{|N}\!\!: N \to X$ is an isometric embedding. A subspace $B \subset X$ of a non-positively curved space $X$ is said to be \emph{locally convex} if every point of $B$ has a neighborhood $N$ such that the geodesics connecting every pair of points of $B \cap N$ are contained in $B$. Note that if $f\!\!:A \to X$ is a locally isometric embedding, then $f(A)$ is locally convex in $X$. Also, if $B \subset X$ is locally convex, then the inclusion $B \hookrightarrow X$ is a locally isometric embedding.

\begin{lem}[product spaces and diagonals]\label{pd}
Let $X$ and $Y$ be non-positively curved spaces. Then
\begin{enumerate}
    \item The metric space $X\times Y$ (equipped with the product metric) is non-positively curved.
    \item  For any choice of $y_0\in Y$ the map $f\!:X\to X\times Y, x\mapsto (x,y_0)$ is a locally isometric embedding. Moreover, the map $f$ induces a monomorphism 
$f_*\!:\pi_1(X)\to \pi_1(X\times Y)$ with image $\pi_1(X)\times 1$.
    \item Assume that $Y\subset X$ and the inclusion of $Y$ into $X$ is a locally isometric embedding. Then the map $g\!\!:Y\to X\times Y, y\mapsto (y,y)$ induces a locally isometric embedding of the scaled metric space $\sqrt{2}Y$ into $X\times Y$ and a group monomorphism 
$g_*\!:~\pi_1(Y)\to \pi_1(X\times Y), h\mapsto (h,h)$ with image $\Delta_{\pi_1(Y)} \leq \pi_1(X) \times \pi_1(Y)$.
\end{enumerate}
\end{lem} 

\begin{proof}
Statement (1) can be seen from Example (3) on page 167 in \cite{MR1744486}. The proof of statement (2) is straightforward from the definition of product metric. The algebra portion of statement (3) is standard. We show that $g$ induces a locally isometric embedding of the scaled metric space $\sqrt{2}Y$ into $X\times Y$ as follows:
$$   d\bigl(g(y_1),g(y_2)\bigr)\; =\; d\bigl((y_1,y_1),(y_2,y_2)\bigr)\; =\; \Bigl( \bigl( d(y_1,y_2) \bigr)^2+\bigl(d(y_1,y_2) \bigr)^2\Bigr)^{1/2} \; =\; \sqrt{2}d(y_1,y_2)\,.
\qedhere$$
\end{proof}




\begin{prop}[\cite{MR1744486}, II.11.6~(2)]
\label{BH2}
Let $X_1$ and $X_2$ be two non-positively curved metric spaces and let $A_1 \subset X_1$ and $A_2 \subset X_2$ be closed subspaces that are locally convex and complete. If $j\!\!: A_1 \to A_2$ is a bijective local isometry, then the quotient of the disjoint union $X=X_1 \sqcup X_2$ by
the equivalence relation generated by $[a_1 \sim j(a_1), \forall a_1 \in A_1 ]$ is also a non-positively curved metric space.
\end{prop}

\section{The group embedding result}

The main result in this section is Proposition~\ref{prop:embed0}, which states that the fundamental group of a particular graph of groups based on a segment of length $(2n+1)$ (so it is an iterated amalgamation of $(2n+2)$ vertex groups) embeds into the fundamental group of a second such graph of groups based on an isomorphic graph. This proposition is used in Section~6 to embed finitely presented groups with superexponential Dehn functions into $\CAT(0)$ groups. 

Proposition~\ref{prop:embed0} is proved using a criterion of Bass (Proposition~\ref{basslemma}) for when a morphism between graphs of groups gives rise to an inclusion between their fundamental groups. In order to apply this criterion we need the preliminary results in Lemma~\ref{lem:bass1} and Lemma~\ref{lem:bass2}. 

The underlying geometry of the construction in Section~5 requires that the vertex groups in the lower graph of groups in Figure~\ref{fig:gog0} are expressed as direct products of the fundamental groups of graphs of groups related to certain subgraphs of the first graph of groups. Therefore, we start this section with a lemma which will be used to establish the inclusions between the corresponding vertex groups in the graphs of groups in Proposition~\ref{prop:embed0}.

\begin{lem}[embeddings of vertex groups]\label{lem:embed1}
Let $(A\rtimes B)$ be a semidirect product and $\theta: B \to C$ be a monomorphism which identifies $B$ with the subgroup $\theta(B)$ of $C$. Define $H$ to be the amalgam 
$$
H \; = \; (A \rtimes B) \ast_{(B \equiv\theta(B))} C. 
$$
Then the map 
$$
\varphi: (A \rtimes B) \, \to \, H \times C\; : ab \mapsto (a,1)(b, \theta(b)) 
$$
defines an embedding of $(A\rtimes B)$ into $H \times C$. 
\end{lem}

\begin{proof}
Given $g \in (A\rtimes B)$, by the normal form in semidirect products there exists unique $a\in A$ and $b \in B$ such that 
$g=ab$. This ensures that $\varphi$ is well-defined as a set map on all of $(A\rtimes B)$ by
$$
\varphi(g) \; =\; \varphi(ab) \; = \; (a,1)(b,\theta(b)).
$$

Next, we show that $\varphi$ is a homomorphism. Given $g_1, g_2 \in (A \rtimes B)$, there exist unique elements $a_1, a_2 \in A$ and $b_1, b_2 \in B$ such that $g_i = a_ib_i$ for $i=1,2$. Therefore, 
\begin{align*}
    \varphi(g_1) \varphi(g_2) &= \varphi(a_1b_1) \varphi(a_2b_2)\\
    &= (a_1b_1, \theta(b_1))(a_2b_2, \theta(b_2))\\
    &= (a_1b_1a_2b_2, \theta(b_1)\theta(b_2))\\
    &= ((a_1b_1a_2b_1^{-1})b_1b_2, \theta(b_1b_2))\\
    &=\varphi((a_1b_1a_2b_1^{-1})(b_1b_2))\\
    &= \varphi(g_1g_2)
\end{align*}
and so $\varphi$ is a group homomorphism. 
The last equality holds because $(a_1b_1a_2b_1^{-1})(b_1b_2)$ is the unique normal form representative of the element $g_1g_2$ in $(A \rtimes B)$.

Let $\pi_1: H \times C \to H: (h,c) \mapsto h$ be the homomorphism which projects onto the first factor. Given $ab \in (A \rtimes B)$ we have $(\pi_1 \circ \varphi)(ab) \; =\; \pi_1(\varphi(ab)) \; =\; 
\pi_1(ab, \theta(b)) \; = \; ab \in H. $
Thus, $\pi_1 \circ \varphi$ is the inclusion map of $(A \rtimes B)$ into $H = (A \rtimes B) \ast_B C$. Therefore, $\varphi$ is injective. 

Note that 
$$
\varphi(A) \; =\;  \{(a, 1) \, |\, a \in A\}
$$
is the subgroup $A \times 1$ of $H \times C$. Also, 
$$
\varphi(B) \; =\; \{(b, \theta(b)) \, |\, b \in B\} \; =\; \{(\theta(b), \theta(b)) \, |\, b \in B\}
$$
where the last equality is true because the relation $\theta(b) = b$ holds in the amalgam $H$ which is the first factor of the direct product. 
Therefore, $\varphi(B)$ is the diagonal subgroup 
$$
\Delta_{\theta(B)} \leq C \times C \leq H \times C. 
$$
In summary, $\varphi$ is a monomorphism which takes $(A \rtimes B)$ isomorphically to the subgroup generated by $A\times 1$ and $\Delta_{\theta(B)}$, 
$$
\langle A \times 1, \, \Delta_{\theta(B)}\rangle \; \subseteq \; H \times C.
\qedhere$$

\end{proof}

The following result tells when a morphism of graphs of groups induces an monomorphism of their fundamental groups. It is Lemma 5.1 from \cite{MR3705143} and is a reformulation of a basic result of \cite{MR1239551}.

\begin{prop}[injectivity for graphs of groups] \label{basslemma}
Suppose $\mathcal{A}$ and $\mathcal{B}$ are graphs of groups such that
the underlying graph $\Gamma_{\mathcal{A}}$ of $\mathcal{A}$ is a
subgraph of the underlying graph of $\mathcal{B}$. Let $A$ and $B$ be
their respective fundamental groups. Suppose that there are injective
homomorphisms $\psi_e\! : A_e \to B_e$ and $\psi_v\! : A_v \to B_v$
between edge and vertex groups, for all edges $e$ and vertices $v$ in
$\Gamma_{\mathcal{A}}$, which are compatible with the edge-inclusion
maps. That is, whenever $e$ has initial vertex $v$, the diagram 
$$
 \begin{tikzcd}[row sep=1.8em,column sep=3em]
  A_e \arrow[r, "i_e"] \arrow[d,"\psi_e"] & A_v \arrow[d, "\psi_v"] \\
 B_e \arrow[r, "j_e"] &                 B_v \\
  \end{tikzcd}
$$
commutes. 

If $j_e(\psi_e(A_e)) = \psi_{v} (A_{v}) \cap j_e(B_e)$ whenever $e$ has initial
vertex $v$, then the induced homomorphism $\psi \!: A \to B$ is
injective. 
\end{prop}

The previous result motivates the following definition. 

\begin{defn}[Bass conditions]
Let $A,B, C, D$ be groups and 
$$
 \begin{tikzcd}
   A \arrow{d}[swap]{\alpha} \arrow{r}{i} & B \arrow{d}{\beta}\\
 C \arrow{r}{j} & D\\
  \end{tikzcd}
$$
be a diagram of monomorphisms. This diagram is said to \emph{satisfy the Bass conditions} if 
\begin{enumerate}
    \item the diagram is commutative, and 
    \item $j(\alpha(A)) \; =\; j(C) \cap \beta(B)$.
\end{enumerate}
We call the second condition the \emph{Bass intersection condition}.
\end{defn}

Lemma~\ref{lem:bass1} and Lemma~\ref{lem:bass2} below describe two situations in which the Bass conditions hold. These will be used repeatedly in the proof of the main embedding result in Proposition~\ref{prop:embed0}.

\begin{lem}[Bass conditions - factor embedding]\label{lem:bass1}
Let $(A\rtimes B)$ be a semidirect product and $\theta: B \to C$ be a monomorphism which identifies $B$ with the subgroup $\theta(B)$ of $C$. Define $H$ to be the amalgam 
$$
H \; = \; (A \rtimes B) \ast_{(B \equiv \theta(B))} C 
$$
and let $\varphi: (A \rtimes B) \to H \times C$ be the monomorphism of Lemma~\ref{lem:embed1}. 

Then the diagram of monomorphisms and inclusions
$$
 \begin{tikzcd}
  (A \rtimes B)  \arrow{d}[swap]{\varphi} & A \arrow{d} \arrow{l}{} \\
 H \times C  & H \arrow{l}{H\times 1} \\
  \end{tikzcd}
$$
satisfies the Bass conditions. 
\end{lem}

\begin{proof}
Note that the left vertical arrow in the diagram is the inclusion map $A \to (A\rtimes B) \to (A\rtimes B) \ast_{(B \cong \theta(B))} C = H$. Therefore, $a \in A$ is mapped to $a \in H$ via the left arrow and this is mapped to $(a,1) \in H\times C$ via the lower arrow. On the other hand $a$ is mapped to $a = a.1 \in (A \rtimes B)$ by the top arrow and $\varphi$ maps this to $(a.1, \theta(1)) = (a,1) \in H \times C$, and so the diagram commutes. 

The image of $A$ in $H \times C$ is $A \times 1$ and from the previous paragraph we have that 
$$A \times 1 \; \subseteq \;(H \times 1) \cap \varphi(A \rtimes B).$$ 
To see the reverse inclusion, let $(g_1, g_2) \in (H \times 1) \cap \varphi(A \rtimes B)$. In particular, $(g_1,g_2) \in H \times 1$ and so $g_2 = 1$. Now $(g_1, 1) \in \varphi(A \rtimes B)$ implies that $(g_1, 1) = (ab, \theta(b))$ for some $ab\in (A \rtimes B)$. But then $\theta(b) = 1$ and so $b=1$ since $\theta$ is a monomorphism. Therefore, $(g_1, g_2) =  (a,1)\in A \times 1$ and so the diagram satisfies the Bass intersection condition. 
\end{proof}

\begin{lem}[Bass conditions - diagonal embedding] \label{lem:bass2}
Let $(A\rtimes B)$ be a semidirect product and $\theta: B \to C$ be a monomorphism which identifies $B$ with the subgroup $\theta(B)$ of $C$. Define $H$ to be the amalgam 
$$
H \; = \; (A \rtimes B) \ast_{(B \equiv \theta(B))} C 
$$
and let $\varphi: (A \rtimes B) \to H \times C$ be the monomorphism of Lemma~\ref{lem:embed1}. Let $\beta$ denote the inverse of the isomorphism $\theta: B \to \theta(B)$. The composition $\theta(B) \stackrel{\beta}{\to} B \hookrightarrow A \rtimes B$ is denoted by $\beta$ too.

Then the diagram of monomorphisms and inclusions 
$$
 \begin{tikzcd}
 \theta(B) \arrow{d}{} \arrow{r}{\beta} &   (A \rtimes B) \arrow{d}{\varphi}\\
 C \arrow{r}[swap]{\Delta_C} & H \times C\\
  \end{tikzcd}
$$
satisfies the Bass conditions. 
\end{lem}

\begin{proof}
An element $c \in \theta(B)$ is mapped to $(c,c)$ via the composition of the right and lower maps. It is sent to $\varphi(\beta(c)) = \varphi(1.\beta(c)) = (1.\beta(c), \theta(\beta(c))) = (\beta(c), c)$ in $H\times C$ via the composition of the top and left maps. However, the relation $\theta(b) = b$ holds in the group $H$ and so $(\beta(c), c) = (\theta(\beta(c)), c) = (c,c)$ and the diagram commutes. 

The image of $\theta(B)$ in $H \times C$ is the subgroup $\Delta_{\theta(B)}$ and from the previous paragraph we have
$$
\Delta_{\theta(B)} \; \subseteq \; \varphi(A\rtimes B) \cap \Delta_C.
$$ 
To see the reverse inclusion, let $(g_1, g_2) \in \varphi(A\rtimes B) \cap \Delta_C$. In particular $(g_1,g_2) \in \Delta_C$ and so $g_2 = g_1$. Now $(g_1, g_1) \in \varphi(A \rtimes B)$ implies that $(g_1, g_1) = (ab, \theta(b))$ for some $ab \in (A \rtimes B)$. Noting that $b=\theta(b)$ in the first factor, this gives $(g_1, g_1) =(a\theta(b), \theta(b))$ which implies $a\theta(b) = \theta(b)$ and so $a=1$. Therefore, 
$(g_1, g_2) = (1.\theta(b), \theta(b)) = (\theta(b), \theta(b))$ is in the image $\Delta_{\theta(B)}$ of $\theta(B)$ in $H \times C$ and so the Bass intersection condition is verified. 
\end{proof}

Here is the main group embedding result of the paper. 

\begin{prop}[graph of groups embedding]\label{prop:embed0}
Suppose we are given \begin{enumerate}
    \item a diagram of groups and inclusions 
    $$
 \begin{tikzcd}
   A_0 \arrow{d}[swap]{\varepsilon_0} \arrow{r} & S \arrow{d}\\
 H_0 \arrow{r} & T\\
  \end{tikzcd}
$$
satisfying the Bass conditions, 
\item a sequence of semidirect products $(A_i \rtimes B_i)$ for $1 \leq i \leq n$, and 
\item isomorphisms $\theta_i: B_i \to A_{i-1}$ for $1 \leq i \leq n$.
\end{enumerate}
Define sequences of groups $H_i$ and $G_i$ for $1\leq i \leq n$ and $L_i$ for $0 \leq i \leq n$ inductively as follows:
\begin{enumerate}
    \item $H_0$ is the group in the diagram above which contains $A_0$ as a subgroup and $$H_i \; =\;  (A_i \rtimes B_i) \ast_{(B_i \equiv _{\theta_i} A_{i-1})} H_{i-1} \quad {\hbox{for $1\leq i \leq n$,}}$$ 
    \item $G_1 = (H_1 \times H_0) \ast_{(\Delta_{H_0} \equiv H_0)} T$ and $$G_i \; = \; (H_i \times H_{i-1}) \ast_{(\Delta_{H_{i-1}} \equiv H_{i-1} \times 1)} G_{i-1}\quad {\hbox{for $2 \leq i \leq n$, and}}$$
    \item $L_0 = S$ contains the subgroup $A_0$ and 
    $$L_i \; = \; (A_i \rtimes B_i) \ast_{(B_i \equiv _{\theta_i} A_{i-1})} L_{i-1}\quad {\hbox{for $1 \leq i \leq n$.}}$$ 
\end{enumerate}
Then the double of $L_n$ over $A_n$ embeds into the double of $G_n$ over $H_n$ where the group $H_n$ includes into $G_n$
 as $H_n \times 1$.
\end{prop}

\begin{proof} Note that each double can be expressed as the fundamental group of a graph of groups whose underlying graph is a segment of length $(2n+1)$. The top line segment of Figure~\ref{fig:gog0} is the underlying graph for the graph of groups description of $L_n \ast_{A_n} L_n$ and the bottom line segment underlies the graph of groups description of $G_n\ast_{H_n} G_n$. The vertex and edge groups are indicated in Figure~\ref{fig:gog0}.

\begin{figure}[ht]
    \centering
    \begin{tikzpicture}
    
     \tikzstyle{every node}=[font=\scriptsize]
     
    \foreach \a in {0, 2}
    {
    \filldraw[black] (0,\a) circle (2pt); 
    \filldraw[black] (2,\a) circle (2pt); 
    \filldraw[black] (6,\a) circle (2pt); 
    \filldraw[black] (8,\a) circle (2pt); 
    \filldraw[black] (12,\a) circle (2pt); 
    \filldraw[black] (14,\a) circle (2pt); 
    
    \draw[thick] (0,\a)--(2,\a)--(3.5,\a); 
    \draw[thick, dashed] (3.6,\a)--(4.4,\a);
    \draw[thick] (4.5,\a)--(6,\a)--(8,\a)--(9.5,\a); 
    \draw[thick, dashed] (9.6,\a)--(10.4,\a);
    \draw[thick] (10.5,\a)--(12,\a)--(14,\a);
    }

    \node[] at (0,2.3) {$S$};
    \node[] at (2,2.3) {$(A_1\rtimes B_1)$};
    \node[] at (6,2.3) {$(A_n\rtimes B_n)$};
    \node[] at (8,2.3) {$(A_n\rtimes B_n)$};
    \node[] at (12,2.3) {$(A_1\rtimes B_1)$};
    \node[] at (14,2.3) {$S$};

    \draw [decorate,decoration={brace,amplitude=5pt,raise=3ex}] (0,2.2) -- (6.5,2.2) node[midway,yshift=2.4em]{$L_n$};
   
    \draw [decorate,decoration={brace,amplitude=5pt,raise=3ex}] (7.5,2.2) -- (14,2.2) node[midway,yshift=2.4em]{$L_n$}; 
    
     \draw [decorate,decoration={brace,mirror,amplitude=5pt}] (-0.1,-0.7) -- (6.8,-0.7) node[midway,yshift=-1.2em]{$G_n$};
   
    \draw [decorate,decoration={brace,mirror, amplitude=5pt}] (7.2,-0.7) -- (14,-0.7) node[midway,yshift=-1.2em]{$G_n$}; 
    
    \node[] at (1, 1.8) {$A_0$};
    \node[] at (3, 1.8) {$A_1$};
    \node[] at (5, 1.8) {$A_{n-1}$};
    \node[] at (7, 1.8) {$A_n$};
    \node[] at (9, 1.8) {$A_{n-1}$};
    \node[] at (11, 1.8) {$A_1$};
    \node[] at (13, 1.8) {$A_0$};
    
    \node[] at (0,-0.3) {$T$};
    \node[] at (2,-0.3) {$(H_1 \times H_0)$};
    \node[] at (6,-0.3) {$(H_n \times H_{n-1})$};
    \node[] at (8,-0.3) {$(H_n \times H_{n-1})$};
    \node[] at (12,-0.3) {$(H_1 \times H_0)$};
    \node[] at (14,-0.3) {$T$};
            
    \node[] at (1,0.3) {$H_0$};
    \node[] at (3,0.3) {$H_1$};
    \node[] at (5,0.3) {$H_{n-1}$};
    \node[] at (7,0.3) {$H_n$};
    \node[] at (9,0.3) {$H_{n-1}$};
    \node[] at (11,0.3) {$H_1$};
    \node[] at (13,0.3) {$H_0$};
    
    \draw[->,  -stealth] (0,1.8)--(0,0.2);
    \draw[->,  -stealth] (2,1.8)--(2,0.2);
    \draw[->,  -stealth] (6,1.8)--(6,0.2);
    \draw[->,  -stealth] (8,1.8)--(8,0.2);
    \draw[->,  -stealth] (12,1.8)--(12,0.2);
    \draw[->,  -stealth] (14,1.8)--(14,0.2);
    
    \draw[->,  -stealth] (1,1.6)--(1,0.5); 
    \draw[->,  -stealth] (3,1.6)--(3,0.5);
    \draw[->,  -stealth] (5,1.6)--(5,0.5);
    \draw[->,  -stealth] (7,1.6)--(7,0.5);
    \draw[->,  -stealth] (9,1.6)--(9,0.5);
    \draw[->,  -stealth] (11,1.6)--(11,0.5);
    \draw[->,  -stealth] (13,1.6)--(13,0.5);
    
    \end{tikzpicture}
    \caption{The morphism between the graphs of groups descriptions of $L_n \ast_{A_n} L_n$ and $G_n\ast_{H_n}G_n$. }
    \label{fig:gog0}
\end{figure}
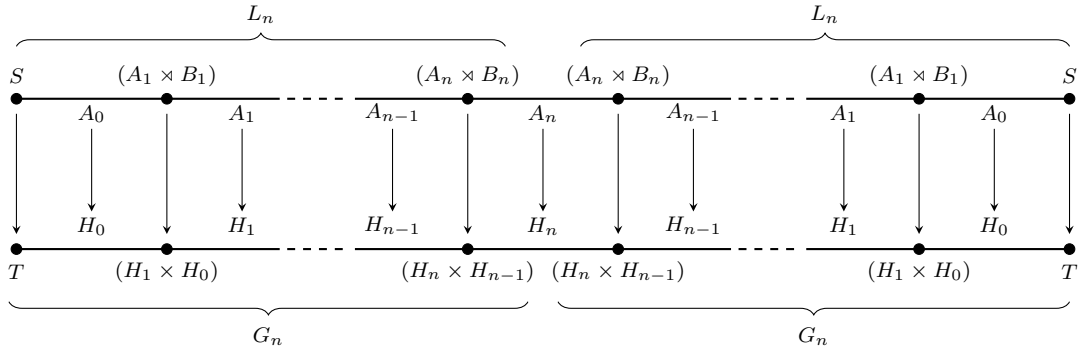

There is an isomorphism of the underlying graphs in Figure~\ref{fig:gog0} and there are inclusion maps $A_i \to H_i = (A_i \rtimes B_i) \ast_{(B_i \equiv_{\theta_i} A_{i-1})} H_{i-1}$ between the corresponding edge groups for $1 \leq i \leq n$. The inclusions $\varepsilon_0: A_0 \to H_0$ are given by hypothesis. Lemma~\ref{lem:embed1} establishes embeddings $\varphi_i: (A_i \rtimes B_i) \to H_i \times H_{i-1}$ between the corresponding vertex groups for $1 \leq i \leq n$. The inclusions $S \to T$ are given by hypothesis.

Next, we verify that the Bass conditions hold for each square of edge group-vertex group inclusions. There are $2(2n+1)$ such squares, two per segment of the diagram in Figure~\ref{fig:gog0}. 
To this end, we label the various maps that are used. The first set of maps are the vertex to edge inclusion maps in the top graph of groups in Figure~\ref{fig:gog0}. 
\begin{itemize}
    \item Let $\gamma_i: A_i \to (A_i \rtimes B_i): a \mapsto a$ denote the inclusion of $A_i$ into the semidirect product.
    \item Let $\beta_i: A_i \to B_{i+1}: a \mapsto \beta_i(a)$ denote the inverse of the isomorphism $\theta_{i+1}: B_{i+1} \to A_i$ given in the hypothesis. Note that $\beta_i$ composed with inclusion embeds $A_i$ into $(A_{i+1} \rtimes B_{i+1})$ as the $B_{i+1}$ subgroup; we denote this composition by $\beta_i$ too.
    \item There is an inclusion $A_0 \to S$ at each end which is left unlabelled. 
\end{itemize}
The next set of maps are the vertex to edge inclusion maps in the bottom graph of groups in Figure~\ref{fig:gog0}. 
\begin{itemize}
    \item Let $\delta_i: H_i \to H_{i+1}\times H_i: h \mapsto (h,h)$ denote the diagonal embedding of $H_i$.
    \item Let $\alpha_i: H_i \to H_i \times H_{i-1}: h \mapsto (h,1)$ denote the inclusion map to the first factor. 
    \item There is an inclusion $H_0 \to T$ at each end which is left unlabelled. 
\end{itemize}
The final set of maps are the vertical edge to edge and vertex to vertex monomorphisms between the two graphs of groups.
\begin{itemize}
    \item Let $\varepsilon_i: A_i \to H_i: a \mapsto a$ denote the inclusion of $A_i$ into $(A_i \rtimes B_i) \subseteq H_i$ for $1 \leq i \leq n$. The inclusion $\varepsilon_0 : A_0 \to H_0$ is given by hypothesis. 
    \item Lemma~\ref{lem:embed1} established that the maps $\varphi_i: (A_i \rtimes B_i) \to H_i \times H_{i-1}: ab \mapsto (ab, \theta_i(b))$ are embeddings for $1 \leq i \leq n$. 
    \item The inclusion $S \to T$ at each end is left unlabelled. 
\end{itemize}

\medskip

\noindent
\emph{Segment labelling.} The diagram in Figure~\ref{fig:gog0} contains $(2n+1)$ segments, corresponding to the $(2n+1)$ edges of each of the underlying graphs. 
We label these segments in order from left to right by the integers $0, 1, \ldots, 2n$. With this labelling, the reflection symmetry about the middle edge sends the $k$th segment to the $(2n-k)$th segment (sending the middle or $n$th segment to itself). 

\medskip

\noindent
\emph{Bass conditions for the middle segment.} 
Consider the following commutative diagram which represents the middle (or 
$n$th) segment of the diagram in Figure~\ref{fig:gog0}. 
$$
 \begin{tikzcd}
  (A_n \rtimes B_n) \arrow{d}{\varphi_n} & A_n \arrow{l}[swap]{\gamma_n} \arrow{d}{\varepsilon_n} \arrow{r}{\gamma_n} & (A_n \rtimes B_n) \arrow{d}{\varphi_n}\\
 (H_n \times H_{n-1}) & H_n \arrow{l}[swap]{\alpha_n} \arrow{r}{\alpha_n} & (H_n \times H_{n-1}) \\
  \end{tikzcd}
$$
By Lemma~\ref{lem:bass1} both the left and right half of this diagram satisfy the two Bass conditions (commuting diagram and intersection condition). To see this, set $A = A_n$, $(A\rtimes B) = (A_n \rtimes B_n)$, $H = H_n$, $C=H_{n-1}$, and $\varphi = \varphi_n$. 

\medskip

\noindent
\emph{Bass conditions for the $k$th segment for $1 \leq k \leq n-1$.} Consider the following commutative diagram which represents the 
$k$th segment (and, by reflection symmetry, the $(2n-k)$th segment) of the diagram in Figure~\ref{fig:gog0}. 
$$
 \begin{tikzcd}
  (A_k \rtimes B_k) \arrow{d}{\varphi_k} & A_k \arrow{l}[swap]{\gamma_k} \arrow{d}{\varepsilon_k} \arrow{r}{\beta_k} & (A_{k+1} \rtimes B_{k+1}) \arrow{d}{\varphi_{k+1}}\\
 (H_k \times H_{k-1}) & H_k \arrow{l}[swap]{\alpha_k} \arrow{r}{\delta_k} & (H_{k+1} \times H_k) \\
  \end{tikzcd}
$$

Lemma~\ref{lem:bass1} implies that the left square of the diagram satisfies the Bass conditions. To see this, take $A = A_k$, $(A \rtimes B) = (A_k \rtimes B_k)$, $H= H_k$, $C = H_{k-1}$, and $\varphi = \varphi_{k}$. 

Lemma~\ref{lem:bass2} implies that the right square of the diagram satisfies the Bass conditions. Here we need to identify $B$ with $B_{k+1}$ and then $\theta(B)$ with $A_k$, and take $(A\rtimes B) = (A_{k+1} \rtimes B_{k+1})$, $H = H_{k+1}$, $C=H_k$, and $\varphi = \varphi_{k+1}$. With this identification the inclusion $\theta(B) \to C$ becomes the inclusion $\varepsilon_k: A_k \to H_k$ and the monomorphism $\beta: \theta(B) \to (A \rtimes B)$ becomes the monomorphism $\beta_k: A_k \to (A_{k+1} \rtimes B_{k+1})$, and Lemma~\ref{lem:bass2} applies. 
\medskip

\noindent
\emph{Bass conditions for the terminal segments.} By symmetry, we only need to consider the $0$th segment. 
$$
 \begin{tikzcd}
  S \arrow{d}{} & A_0 \arrow{l}{} \arrow{d}{\varepsilon_0} \arrow{r}{\beta_0} & (A_1 \rtimes B_1) \arrow{d}{\varphi_1}\\
 T & H_0 \arrow{l}{} \arrow{r}{\delta_0} & (H_1 \times H_0) \\
  \end{tikzcd}
$$
The proof that the right square satisfies the Bass conditions follows from Lemma~\ref{lem:bass2} exactly as the proof for the right square in the $k$th segment above. The left square satisfies the Bass conditions by hypothesis. 

\medskip

\noindent
\emph{Conclusion.} 
Finally, Proposition~\ref{basslemma} implies that the fundamental group of the top graph of groups injects into the fundamental group of the bottom graph of groups; that is, the double 
$L_n\ast_{A_n}L_n$ embeds into the double $G_n\ast_{H_n} G_n$. 
\end{proof}

\begin{rem}[applications]\label{rem:apps}
There are two main applications of Proposition~\ref{prop:embed0}. For the first application, $A_0 = F_2$ is a subgroup of a snowflake group $S$ from \cite{MR3705143}. The snowflake group is a subgroup of a 6--dimensional $\CAT(0)$ group $T$ and $H_0 = (F_2 \rtimes \Z)$ is a convex subgroup of $T$. 
The fact that the diagram of group inclusions involving $A_0$, $H_0$, $S$, and $T$ satisfies the Bass conditions is established in Lemma~5.3 of \cite{MR3705143}. 

For the second application, $H_0 = (F_k \rtimes \Z)$ is a $\CAT(0)$ free-by-cyclic group and $A_0 = F_k$. The groups $S$ and $T$ are defined to be $S = (F_k \rtimes \Z)$ and $T = (F_k \rtimes \Z) \times \Z$. Note that $T$ is a $3$--dimensional $\CAT(0)$ group. The embedding $\varphi_0: S \to T$ has image the subgroup of $T$ generated by $(F_k \times 1) \cup \Delta_{\Z}$ and the embedding $H_0 \to T$ has image the subgroup $H_0 \times 1$. 
The fact that the diagram of group inclusions involving $A_0$, $H_0$, $S$, and $T$ satisfies the Bass conditions follows directly from Lemma~\ref{lem:bass1}. 
\end{rem}

\begin{rem}[normality]\label{rem:BS}
The proofs of Lemma~\ref{lem:bass1} and Proposition~\ref{prop:embed0} rely on the fact that each $A_i$ is normal in $(A_i \rtimes B_i)$. This appears in the use of normal forms for semidirect products in the proofs of the embedding results. It is interesting to see what goes wrong when one tries to emulate these embedding results using, for example, ascending HNN extensions instead of semidirect products.

Consider one of the simplest ascending HNN extensions, the Baumslag-Solitar group
$$
BS(1,2) \; =\; \langle a, t \, |\, tat^{-1} = a^2\rangle.
$$
There is an embedded copy of $BS(1,2)$ in the direct product $BS(1,2) \times \langle t \rangle$ which is analogous to the $\varphi(A \rtimes B)$ subgroup of Lemma~\ref{lem:embed1};
namely, the group generated by 
$$\langle (a,1),  (t,t)\rangle. 
$$
As in Lemma~\ref{lem:embed1}, this is the group generated by the subgroup $\langle a \rangle \times 1$ and the diagonal subgroup $\langle (t,t)\rangle$. 

However, this embedding does not give rise to an embedding of the double of $BS(1,2)$ over $\langle a\rangle$ into the double of 
$BS(1,2) \times \Z$ over $BS(1,2) \times 1$. 
This is because the group $BS(1,2) \times 1$ is normal in $BS(1,2) \times \Z$ and so the intersection 
$$
(BS(1,2) \times 1) \cap \langle (a,1), (t,t)\rangle
$$
is the normal closure of the cyclic subgroup $\langle (a,1)\rangle$ in $\langle (a,1), (t,t)\rangle = BS(1,2)$. This is isomorphic to $\Z[1/2]$ and is not finitely generated. For this example, the analogue of Lemma~\ref{lem:bass1} does not hold and the proof of Proposition~\ref{prop:embed0} breaks down. 

In this case, the double is isomorphic to $BS(1,2) \ast_{\Z[1/2]}BS(1,2)$ which is not finitely presented. It is a quotient of the double of $BS(1,2)$ over $\langle a \rangle$.

One may wish to explore this example by using new letters (say $u$ and $v$) for the copies of $\Z$. With this notation, the double of $BS(1,2) \times \Z$ over $BS(1,2) \times 1$ is the group 
$$\langle a, t, u, v \, |\, tat^{-1}= a^2, [u,a]=[u,t] =[v,a]=[v,t]=1\rangle . $$
The copies of $BS(1,2)$ in each $BS(1,2) \times \Z$ ``side'' are 
$$
\langle a, tu \rangle  \quad {\hbox{and}} \quad \langle a, tv \rangle . 
$$
The group 
$$
\langle a, tu, tv\rangle 
$$
is not isomorphic to the double of $BS(1,2)$ over $\langle a \rangle$. Instead of an embedding, one obtains a quotient map from 
$$
BS(1,2) \ast_{\langle a\rangle} BS(1,2) \; = \; \langle a, s_1, s_2 \, |\, s_1as_1^{-1} = a^2 = s_2as_2^{-1}\rangle
$$ 
to 
$$\langle a, tu, tv \rangle \; =\; 
BS(1,2)\ast_{Z[1/2]}BS(1,2) \; \leq \; (BS(1,2) \times \langle u\rangle) \ast_{BS(1,2) \times 1} (BS(1,2) \times \langle v \rangle)
$$
sending $a \mapsto a$, $s_1 \mapsto tu$, $s_2 \mapsto tv$.
For example, $[(tv)(tu)^{-1}, a] = 1$ in this group, but the corresponding word 
$[s_1s_2^{-1}, a]$ in the double 
$\langle a, s_1, s_2 \, | \, s_1as_1^{-1} = a^2,\,  s_2as_2^{-1} = a^2\rangle$ is not trivial.

\end{rem}

\begin{rem}[no short cuts]
In Proposition~\ref{prop:embed0}, one might be tempted to embed the double of $H_n$ over $A_n$ into the double $(H_n \times H_{n-1}) \ast_{H_n} (H_n \times H_{n-1})$. This seems plausible because $H_n$ embeds into $H_n \times H_{n-1}$ with image the subgroup
$$
\langle A_n \times 1, \Delta_{H_{n-1}}\rangle. 
$$
We see this as follows. Given the setup of Lemma~\ref{lem:embed1} one can consider the diagram 
$$
 \begin{tikzcd}
  (A \rtimes B) \arrow{d}{\varphi} & B \arrow{l}{} \arrow{d}{(b, \theta(b))} \arrow{r}{\theta} & C \arrow{d}{\Delta_C}\\
 (A \rtimes B) \times C & B \times C \arrow{l}{} \arrow{r}[swap]{(\theta(b), c)} & C \times C \\
  \end{tikzcd}
$$
The right diagram commutes since 
$b \mapsto (b,\theta(b)) \mapsto (\theta(b), \theta(b))$ gives the same result as 
$b \mapsto  \theta(b) \mapsto (\theta(b), \theta(b))$. Also 
$$
\{(c,c) \, |\, c \in C\} \cap \{(\theta(b), c) \, |\, b \in B, c\in C\} \; =\; 
\{(\theta(b), \theta(b)) \, |\, b \in B\}
$$
is the image of $B$ in $C \times C$ and so the Bass intersection condition holds. 

The left diagram commutes since $b \mapsto (b, \theta(b)) \mapsto (b, \theta(b))$
gives the same result as 
$b \mapsto b = 1.b \mapsto (1.b, \theta(b))$. 
Furthermore, 
$$
B \times C \cap \{(ab, \theta(b)) \, |\, ab \in (A\rtimes B)\} \; =\; \{1.b, \theta(b)) \, | \, b \in B\}
$$
is the image of $B$ in $(A\rtimes B) \times C$ and so the Bass intersection condition holds. 

By Proposition~\ref{basslemma} and using the notation of Lemma~\ref{lem:embed1} $H = (A\rtimes B)\ast_{(B \cong \theta(B))} C$ embeds into $H \times C$ with image the subgroup generated by $A\times 1 \cup \Delta_C$.

As a special application, we have that $H_n = (A_n \rtimes B_n) \ast_{(B_n \cong \theta_n(B_n))} H_{n-1}$ embeds into $H_n \times H_{n-1}$ with image the subgroup 
$$
\langle A_n \times 1, \, \Delta_{H_{n-1}} \rangle . 
$$

Nevertheless, the double 
$H_n \ast_{A_n} H_n$ does not embed into the double $(H_n \times H_{n-1}) \ast_{H_n} (H_n \times H_{n-1})$. 
This is because the Bass intersection condition 
$$
A_n \times 1 \; = \; (H_n \times 1) \cap \langle A_n \times 1, \, \Delta_{H_{n-1}} \rangle
$$
fails to hold in the case $n>1$. The subgroup $(H_n \times 1)$ is normal in $(H_n \times H_{n-1})$ and so the intersection $(H_n \times 1) \cap \langle A_n \times 1, \, \Delta_{H_{n-1}} \rangle $ is the normal closure of $A_n \times 1$ in $\langle A_n \times 1, \, \Delta_{H_{n-1}} \rangle $. For 
$n>1$ this normal closure is strictly larger than $(A_n \times 1)$ and so the induced map is not an embedding. 


\end{rem}

\section{Dehn functions of the finitely presented subgroups}

The purpose of this section is compute the Dehn functions of the subgroups $L_n \ast_{A_n}L_n$ of Proposition~\ref{prop:embed0} for various choices of the terminal vertex groups $S$. 
In order to do this we need to first compute the distortion of $A_n$ in $L_n$. A key ingredient is the fact that the definition of $L_n$ involves an iterated amalgamation of hyperbolic free-by-free groups. 

The first lemma below will be used inductively in Proposition~\ref{prop:distortion-short} to establish the upper bound on the distortion of $A_n = F_{k_{n+1}}$ in $L_n$. Recall that a function $f\!\!\!: [0,\infty) \to [0,\infty)$ is super-additive if $f(x+y) \geq f(x) + f(y)$ for all $x, y \in [0,\infty)$. 

\begin{lem}[amalgam distortion]
\label{ud}
Let $G=(F\rtimes K)*_K H$ be a group amalgamation where the groups $F$, $K$, and $H$ are all finitely generated. Assume that the distortion function $\Dist^H_K$ is dominated by a non-decreasing, super-additive function $f$. Then the distortion function $\Dist^G_F$ is dominated by the composite $\Dist^{F\rtimes K}_F \circ f$. 
\end{lem}

\begin{proof}
Let $G_1=F\rtimes K$ and let $S_F$, $S_K$, and $S_H$ be finite generating sets of $F$, $K$, and $H$ respectively. Assume that $S_K$ is a subset of $S_H$. Then $S_{G_1}=S_F\cup S_K$ and $S_G=S_F\cup S_H$ are finite generating sets of $G_1$ and $G$ respectively and $S_{G_1}$ is a subset of $S_G$. We denote $d_F$, $d_K$, $d_H$, $d_{G_1}$, and $d_G$ the word metrics on the corresponding groups with respect to the given choice of generating sets.

We observe that the homomorphism $G\to H$ taking all generators of $S_F$ to the identity and each generator of $S_H$ to itself shows that $H$ is a retract of $G$ and so are isometrically embedded subgroups. Therefore,
$$\Dist_K^G=\Dist_K^H\preceq f\,.$$
This implies that there are positive integers $C$ and $D=C^2$ such that for each $x\geq 1$ we have $$\Dist_K^G(x)\leq Cf(Cx)\leq f(Dx)\,.$$
We note that the second inequality above comes from the super-additive property of $f$ and $D=C^2$. We also increase $C$ so that $f(Dx)\geq x \text{ for each } x\geq 1\,.$ 

We now prove that $\Dist_{G_1}^G(x)\leq f(Dx)$ for all $x\geq 1$. Indeed, let $b$ be an arbitrary group elements in $G_1$ such that $d_{G}(1,b)\leq x$. Therefore, there is a word $w=u_1v_1u_2v_2\cdots u_nv_n$ in $S_G$ with the length at most $x$ such that
\begin{enumerate}
    \item Each $u_i$ is a word in $S_{G_1}$ and $u_1$ is possibly empty.
    \item Each $v_i$ is not necessarily a word in $S_K$ but it represents a group element $k_i$ in $K$ and $v_n$ is possibly empty. 
\end{enumerate}
By the construction we note that $d_G(1,k_i)\leq \ell(v_i)$. Therefore, 
$$d_K(1,k_i)\leq \Dist_K^G\bigl(\ell(v_i)\bigr)\leq f\bigl(D\ell(v_i)\bigr)\,.$$
This implies that there is a word $v'_i$ in $S_K$ (therefore also in $S_{G_1}$) representing $k_i$ such that $$\ell(v'_i)=d_K(1,k_i)\leq f\bigl(D\ell(v_i)\bigr)\,.$$ Therefore, $w'=u_1v'_1u_2v'_2\cdots u_nv'_n$ is a word in $S_{G_1}$ and represents the group element $b$. Thus,
\begin{align*}
   d_{G_1}(1,b)&\leq\ell(w')\\&\leq \ell(u_1)+\ell(v'_1)+\ell(u_2)+\ell(v'_2)+\cdots+\ell(u_n)+\ell(v'_n)\\&\leq f\bigl(D\ell(u_1)\bigr)+f\bigl(D\ell(v_1)\bigr)+f\bigl(D\ell(u_2)\bigr)+f\bigl(D\ell(v_2)\bigr)+\cdots+f\bigl(D\ell(u_n)\bigr)+f\bigl(D\ell(v_n)\bigr)\\&\leq f\bigl(D(\ell(u_1)+\ell(v_1)+\ell(u_2)+\ell(v_2)+\cdots+\ell(u_n)+\ell(v_n))\bigr)\\&\leq f(Dx) \,. 
\end{align*}
This implies that $\Dist_{G_1}^G(x)\leq f(Dx)$ for all $x\geq 1$. Therefore,
$$\Dist_F^G(x)\leq (\Dist_F^{G_1}\circ \Dist_{G_1}^G)(x) \leq (\Dist_F^{G_1}\circ f)(Dx) \text{ for each } x\geq 1$$
and so $\Dist_F^G\preceq \Dist_F^{G_1}\circ f$.
\end{proof}

The next two lemmas will be used in Proposition~\ref{prop:distortion-short} to establish lower bounds for the distortion of $A_n = F_{k_{n+1}}$ in $L_n$. 

\begin{lem}[Lemma 11.64 in \cite{DK18}]
\label{div}

Let $X$ be a hyperbolic space. Then there is a constant $\alpha \in (0,1)$ depending on the hyperbolicity constant of $X$ such that the following hold. If $[x,y]$ is a geodesic of length $2r$ and $m$ is its midpoint, then every path joining
$x$, $y$ outside the ball $B(m, r)$ has length at least $2^{\alpha (r-1)}$.
\end{lem}

\begin{lem}[distortion in hyperbolic free-by-free]
\label{ld}
Let $F_\ell\rtimes F_k$ be a hyperbolic free-by-free group. Let $d_{{F_\ell}}$ and $d_{{F_k}}$ be the word metrics with respect to finite generating sets $S_\ell$ and $S_k$ of $F_\ell$ and $F_k$ respectively. Given $1 \not= b \in F_\ell$ there exists a constant $A>1$ such that if $g \in F_k$ with $d_{F_k}(1,g)$ sufficiently large, then $$d_{{F_\ell}}(1,gbg^{-1})\geq A^{d_{F_k}(1,g)} \,.$$
\end{lem}

\begin{proof}
Let $G=F_\ell\rtimes F_k$. Then $S_G=S_\ell\cup S_k$ is a generating set of $G$. We denote $d_G$ the word metric on $G$ with respect to the finite generating set $S_G$ and assume that $(G,d_G)$ is a $\delta$--hyperbolic space for some $\delta >1$. The group monomorphism sending all elements in $S_\ell$ to the identity and sending each element in $S_k$ to itself shows that $F_k$ is isometrically embedded into $G$ with the given word metrics.

We claim that $bF_kb^{-1}\cap F_k=\{1\}$. Assume to the contrary that $bF_kb^{-1}\cap F_k$ is a non-trivial group. Then there are nontrivial two group elements $a_1$ and $a_2$ in $F_k$ such that $ba_1b^{-1}=a_2$. This implies that $a_1^{-1}a_2=(a_1^{-1}b a_1)b^{-1}$ is a group element in the trivial group $F_k\cap F_\ell$. Therefore $a_1=a_2$ which commutes to $b\in F_\ell$. This implies that $a_1$ belong to the centralizer $C(b)$ of $b$ which is virtually cyclic (see \cite{MR1086648}). Therefore, there are some non-zero integers $p$ and $q$ such that $a_1^p=b^q\in F_k\cap F_\ell=\{1\}$. Since $G$ is torsion free, the group element $a_1$ must be trivial which is a contradiction.
Thus, $$bF_kb^{-1}\cap F_k=\{1\}\,.$$ By Proposition 9.4 in \cite{Hruska10} there is a constant $D>0$ such that $$\diam \bigl(N_{5\delta}(F_k)\cap N_{5\delta}(bF_k)\bigr)\leq D\,.$$
Therefore, $$\diam \bigl(N_{5\delta}(F_k)\cap N_{5\delta}((gb)F_k)\bigr)=\diam g\bigl(N_{5\delta}(F_k)\cap N_{5\delta}(bF_k)\bigr)\leq D\,.$$

Let $n=d_{F_k}(1,g)$. Let $\beta_1$ (resp. $\beta_2$) be the geodesic in the Cayley graph $\Gamma(G,S_G)$ connecting $1$ and $g$ (resp. connecting $gb$ and $gbg^{-1}$) with edges labeled by elements in $S_k$. Then vertices of $\beta_1$ (resp. $\beta_2$) are group elements in $F_k$ (resp. $(gb)F_k$). Let $e_1$ be an edge label by $b\in S_\ell$ connecting $g$ and $gb$. Let $\beta_3$ be a geodesic in $\Gamma(G,S_G)$ connecting $1$ and $gbg^{-1}$. Since $(G,d_G)$ is $\delta$--hyperbolic, there are vertices $m$ in $\beta_3$, $x\in\beta_1$, and $y\in (e_1\cup\beta_2)$ such that 
$$d_G(m,x)\leq 2\delta \text{ and } d_G(m,y)\leq 2\delta\,.$$
This implies that $$d_G(x,(gb)F_k)\leq d_G(x,m)+d(m,y)+d(y,(gb)F_k)\leq 2\delta+2\delta+1< 5\delta\,.$$
Hence $x\in N_{5\delta}(F_k)\cap N_{5\delta}((gb)F_k)$. Also, $g\in N_{5\delta}(F_k)\cap N_{5\delta}((gb)F_k)$. This implies that $$d_G(x,g)\leq\diam \bigl(N_{5\delta}(F_k)\cap N_{5\delta}((gb)F_k)\bigr)\leq D\,.$$ Therefore, $$d_G(g,m)\leq d_G(g,x)+d_G(x,m)\leq D+2\delta.\,.$$

Let $\gamma$ be the path in the Cayley graph $\Gamma(G,S)$ which connects $1$ and $gbg^{-1}$ and traces the shortest word in $S_\ell$ representing the element $gbg^{-1}$. Then $\ell(\gamma)=d_{F_\ell}(1,gbg^{-1})$. Moreover, each vertex of $v$ of $\gamma$ is an element $a\in \F_\ell$. The group monomorphism sending all elements in $S_\ell$ to the identity and sending each element in $S_k$ to itself shows that $d_G(g,v)\geq d_{F_k}(g,1)=n$. This implies that $\gamma$ lies outside the open ball $B(g,n)$. Also $d_G(g,m)\leq D+2\delta$. The path $\gamma$ lies outside the open ball $B(m,n-D-2\delta)$. Here we assume that $n>D+2\delta$.

 Let $z_1$ (resp. $z_2$) be the point in the geodesic segment $[1,m]$ (resp. $[gbg^{-1},m]$) of $\beta_3$ such that the length of $[z_1,m]$ (resp. $[z_2,m]$) is exactly $n-D-2\delta$. Let $\gamma_1$ (resp. $\gamma_2)$ be the geodesic subsegment of $[1,m]$ (resp. the subgeodesic segment of $[gbg^{-1},m]$) connecting $1$ and $z_1$ (resp. connecting $gbg^{-1}$ and $z_2$). Therefore,
\begin{align*}
  \ell(\gamma_1)&= d_G(1,z_1)\\&=d_G(1,m)-d_G(z_1,m)\\&\leq \bigl(d_G(1,g)+d_G(g,m)\bigr)- (n-D-2\delta)\\&\leq (n+D+2\delta)-(n-D-2\delta)\\&\leq 2D+4\delta\,.
\end{align*}
Similarly, we also have
\begin{align*}
  \ell(\gamma_2)&= d_G(gbg^{-1},z_2)\\&=d_G(gbg^{-1},m)-d_G(z_2,m)\\&\leq \bigl(d_G(gbg^{-1},gb)+d_G(gb,g)+d_G(g,m)\bigr)- (n-D-2\delta)\\&\leq \bigl(d_G(g^{-1},1)+1+d_G(g,m)\bigr)- (n-D-2\delta)\\&\leq (n+1+D+2\delta)-(n-D-2\delta)\\&\leq 2D+4\delta+1\,.
\end{align*}

The subsegment $[z_1,z_2]$ of $\beta_3$ is a geodesic of length $2\bigl(n-D-2\delta\bigr)$ and $m$ is its midpoint. By Lemma~\ref{div} the path $\bar{\gamma}=\gamma_1\cup \gamma \cup \gamma_2$ joining
$z_1$, $z_2$ outside the open ball $B(m,n-D-2\delta)$ has length at least $2^{\alpha (n-D-2\delta)}$ where the constant $\alpha>0$ only depends on the hyperbolicity constant of $(G,d_G)$. Therefore,
\begin{align*}
  d_{F_\ell}(1,gbg^{-1})&=\ell(\gamma)\\&=\ell(\bar{\gamma})-\ell(\gamma_1)-\ell(\gamma_2)\\&\geq 2^{\alpha (n-D-2\delta)}-(2D+4\delta)-(2D+4\delta+1)\\&\geq (\sqrt{2})^{\alpha n}  
\end{align*}
for $n$ sufficiently large. Therefore, $A=(\sqrt{2})^{\alpha}>1$ is the desired constant.
\end{proof}

The following proposition establishes the Dehn function of the subgroups $L_n \ast_{A_n} L_n$ of Proposition~\ref{prop:embed0}. 

\begin{prop}[distortion and Dehn functions in amalgams]\label{prop:distortion-short}
Let $n$ be a positive integer and let 
$F_{k_{i+1}} \rtimes F_{k_i}$ be a hyperbolic free-by-free group for $1 \leq i \leq n$. Let $S$ be a finitely generated group containing a free subgroup $F_{k_1}$ such that $\Dist_{F_{k_1}}^S$ is equivalent to a non-decreasing, super-additive function $f$. As in Proposition~\ref{prop:embed0} define groups $L_i$ for $0\leq i\leq n$ inductively by $L_0=S$ and 
$L_i \; =\; (F_{k_{i+1}} \rtimes F_{k_i}) \ast_{F_{k_i}} L_{i-1}$. 
Then the distortion of $F_{k_{n+1}}$ in $L_n$ is equivalent to $\exp^{(n)}(f(x))$. 

Moreover, if the Dehn function of $L_n$ is dominated by a polynomial function, then the Dehn function of the double $L_n \ast_{F_{k_{n+1}}}L_n$ is equivalent to $\exp^{(n)}(f(x))$. 

\end{prop}

\begin{proof}

We first prove the upper bound for the distortion $\Dist_{F_{k_{n+1}}}^{L_n}$. First of all, note that for each $i\in\{1,2,\cdots,n\}$ the distortion of $F_{k_{i+1}}$ in $F_{k_{i+1}} \rtimes F_{k_i}$ is equivalent to $e^x$. The exponential lower bound for $\Dist^{F_{k_{i+1}} \rtimes F_{k_i}}_{F_{k_{i+1}}}$ follows from Lemma~\ref{ld} and the exponential upper bound can be seen from Exercise 6.19 in part III.$\Gamma$ in \cite {MR1744486}. Next, apply Lemma~\ref{ud} inductively to obtain the upper bound of $\exp^{(n)}(f(x))$ for $\Dist_{F_{k_{n+1}}}^{L_n}$. The induction works because $f$ is non-decreasing and super-additive by hypothesis and each $\exp^{(i)}(f(x))$ is also non-decreasing and super-additive.

We now prove the lower bound of the distortion. For $1\leq i\leq n+1$ let $S_i$ be a generating set of $F_{k_i}$. 
These choices of generating sets induces the word metrics on $F_{k_1}, F_{k_2}, \cdots, F_{k_{n+1}} $ and we denote them $d_{F_{k_1}}, d_{F_{k_2}}\cdots, d_{F_{k_{n+1}}}$ respectively. We fix a finite generating set $T$ of $S$ that contains the finite generating set $S_1$ of $F_{k_1}$. Then $S_{L_n}=\bigl(\bigcup S_i\bigr)\bigcup T$ generates $L_n$ and we let $d_{L_n}$ be the word metric on $L_n$ induced by $S_{L_n}$. We also denote $d_S$ the word metric on the group $S$ induced by the generating set $T$.

For each $i\in\{2,3,\cdots,n+1\}$ choose $b_i$ be an element in the generating set $S_{i}$ of $F_{k_i}$. By Lemma~\ref{ld} there is a constant $A>1$ such that for each $i\in\{1,2,\cdots,n\}$ if $g_i$ is a group element in $F_{k_i}$ with $d_{F_{k_i}}(1,g_i)$ sufficiently large then $$d_{{F_{k_{i+1}}}}(1,g_ib_{i+1}g_i^{-1})\geq A^{d_{F_{k_i}}(1,g_i)} \,.$$ Define the function $F(x)=A^x$ for each $x\geq 1$. Then we observe that two functions $F^{(n)}(f(x))$ and $\exp^{(n)}(f(x))$ are equivalent. Therefore, it is sufficient to show that $\Dist^{L_n}_{F_{k_{n+1}}}$ dominates the function $F^{(n)}(f(x))$. 

Since $\Dist^S_{F_{k_1}}$ is equivalent to $f(x)$, there is a positive integer $D>1$ such that $$f(x)\leq D \Dist^S_{F_{k_1}}(Dx) \text{ for each $x\geq 1$}\,.$$ Let $C=D^2>1$. Then we observe that
$$\Dist^S_{F_{k_1}}(Cx)= \Dist^S_{F_{k_1}}(D^2x)\geq (1/D)f(Dx)\geq f(x)\text{ for each $x\geq 1$}\,.$$
We note that the last inequality holds due to the fact $f(x)$ is super-additive. 

Let $x\geq 1$ be an arbitrary number. Let $g_1$ be a group element in $F_{k_1}$ such that $$d_S(1,g_1)\leq Cx \text{ and } d_{F_{k_1}}(1,g_1)=\Dist^S_{F_{k_1}}(Cx)\geq f(x)\,.$$ For each $i\in\{2,3,\cdots,n+1\}$ we define $g_i=g_{i-1}b_ig_{i-1}^{-1}$ by induction on $i$. Then each $g_{i}$ is a group element in $F_{k_{i}}$. By using an argument on $i$ we can prove that for each $i\in \{0,1,2,\cdots,n\}$ we have $$d_{{F_{k_{i+1}}}}(1,g_{i+1})\geq F^{(i)}(f(x)) \text{ for $x$ sufficiently large}\,.$$
In particular, $$d_{{F_{k_{n+1}}}}(1,g_{n+1})\geq F^{(n)}(f(x)) \text{ for $x$ sufficiently large}\,.$$

By using an argument on $i$ again we can also prove for each $i\in \{0,1,2,\cdots,n\}$ we have $$d_{L_n}(1,g_{i+1})\leq (3^iC)x\,.$$
In particular, $d_{L_n}(1,g_{n+1})\leq (3^nC)x\,.$ Therefore, 
$$\Dist^{L_n}_{F_{k_{n+1}}}\bigl((3^nC)x\bigr)\geq d_{{F_{k_{n+1}}}}(1,g_{n+1})\geq F^{(n)}(f(x)) \text{ for $x$ sufficiently large}\,.$$
This implies that $\Dist^{L_n}_{F_{k_{n+1}}}$ dominates the function $F^{(n)}(f(x))$. Therefore, the distortion of $F_{k_{n+1}}$ in $L_n$ is equivalent to $\exp^{(n)}(f(x))$.

We now prove that the Dehn function of $L_n \ast_{F_{k_{n+1}}} L_n$ is $\exp^{(n)}(f(x))$. Let $D_n$ be the double of $L_n$ over $F_{k_{n+1}}$ and denote the Dehn function of $D_n$ by $\delta_{D_n}$. 
By hypothesis, the Dehn function $\delta_{L_n}$ of $L_n$ is dominated by a polynomial function $x^d$. By Theorem~\ref{BH1} we have
$$\Dist^{L_n}_{F_{k_{n+1}}}(x)\preceq \delta_{D_n}(x)\preceq x\bigl(\delta_{L_n}\circ \Dist^{L_n}_{F_{k_{n+1}}}\bigr)(x)\,,$$
which implies that
$$\exp^{(n)}(f(x)) \preceq \delta_{D_n}(x)\preceq x\bigl(\exp^{(n)}(f(x))\bigr)^d\leq \bigl(\exp^{(n)}(f(x))\bigr)^{d+1}\,.$$
Note that the two functions $\exp^{(n)}(f(x))$ and $\bigl(\exp^{(n)}(f(x))\bigr)^{d+1}$ are equivalent (the exponent $d+1$ is absorbed into the super-additive function $\exp^{(n-1)}(f(x))$). Then the Dehn function of the double of $L_n$ over $F_{k_{n+1}}$ is equivalent to $\exp^{(n)}(f(x))$.
\end{proof}

As we shall see in Section~5, the hypothesis in the previous proposition that $L_n$ has a polynomial isoperimetric inequality holds automatically in the case that the terminal vertex group $S$ is a $\CAT(0)$ $F_k \rtimes \Z$ group. The next proposition shows that $L_n$ has a polynomial isoperimetric inequality in the case where the terminal vertex group $S$ is a snowflake group. 

The proof of the next proposition uses the machinery of singular disk diagrams and combinatorial 2--complexes. The reader may wish to refer to Section~2 of \cite{2020arXiv201200730B} for background on combinatorial complexes and singular disk fillings of combinatorial loops; specifically, Definitions~2.4, 2.9, 2.10, and 2.11. This result is similar to Theorem~5.1 of \cite{MR1781853} and to Theorem~4 of \cite{MR2691093}.

\begin{prop}[Dehn function of amalgams]
\label{polydehn}
Let $G=A*_C B$ be a group amalgamation where $A$ and $B$ are finitely presented groups and $C$ is a finitely generated group whose distortion $\Dist_C^G(x)$ is equivalent to a super additive function. Then $$\delta_G(x)\preceq x\max\{\delta_A\bigl(x\Dist^G_C(x)\bigr),\delta_B\bigl(x\Dist^G_C(x)\bigr)\}\,,$$
where $\delta_A$, $\delta_B$, and $\delta_G$ are Dehn functions of $A$, $B$, and $G$ respectively.
\end{prop}

\begin{proof}
Pick finite generating sets $S_A$, $S_B$, and $S_C$ for the groups $A$, $B$, and $C$ respectively with the property that the image of $S_C$ in $A$ is contained in $S_A$ and the image of $S_C$ in $B$ is contained in $S_B$. There are finite presentations $\langle S_A \, |\, R_A\rangle$ of $A$ and $\langle S_B \, |\, R_B\rangle$ of $B$ and a possibly infinite presentation $\langle S_C \, |\, R_C\rangle$ of $C$. 
Let $P_A$, $P_B$, $P_C$ denote the corresponding (combinatorial) presentation 2--complexes. There are combinatorial maps $f_A: P_C \to P_A$ inducing the inclusion $C \to A$ and $f_B: P_C \to P_B$ inducing the inclusion $C \to B$. 

The space 
$$
[P_A \; \sqcup \; (P_C \times [0,1]) \; \sqcup P_B]/\sim
$$
where $\sim$ is generated by $(x,0) \sim f_A(x)$ and $(x,1) \sim f_B(x)$ for all $x \in P_C$ is a 3--dimensional cell complex with fundamental group $A\ast_CB$. This complex has finite 2--skeleton. The 1--skeleton is the union of a wedge of $|S_A|$ circles and a wedge of $|S_B|$ circles whose base vertices are connected by an edge (the image of the base vertex of $P_C$ times $[0,1]$). We denote this edge by $e$ and orient it from 0 to 1. The 2--cells are indexed by the disjoint union of the sets $R_A$, $R_B$, and $S_C$; the last set indexing square 2--cells with boundary $\bar{c}ec\bar{e}$ for $c \in S_C$. Therefore, the 2--skeleton of the universal cover is a geometric model for $A\ast_CB$. Denote this 2--skeleton by $X$. 

We need to bound the area of an edgepath loop in $X$ as a function of its length. 
Let $\gamma$ be a combinatorial edgepath loop in $X$. A positively oriented edge $e$ in $\gamma$ corresponds to the path leaving a copy of $\widetilde{P}_A$ in $X$ and entering a copy of $\widetilde{P}_B$. The path $\gamma$ will travel in the 1--skeleton of this copy of $\widetilde{P}_B$ and then exit along a (possibly different) coset of $C$ via an edge labelled $\overline{e}$. Similarly, an edge $\overline{e}$ in $\gamma$ is followed by an edge $e$. Therefore, the orientations of the edges labelled $e$ alternate along $\gamma$. We match the $e$ and $\bar{e}$ edges of $\gamma$ in pairs as follows. 

One can think of $\gamma$ as a combinatorial map $\gamma:S^1 \to X$ for some cell structure on $S^1$; that is, a cellular map which sends open 1--cells of $S^1$ to open 1--cells of $X$. Pull back the edge labels in $X$ to give a labelling of the edges of the domain $S^1$ of 
$\gamma$.  Since $X$ is simply-connected, there is a continuous map $g: D^2 \to X$ such that $g|_{\partial D^2} = \gamma$. Homotope $g$ rel boundary so that it is transverse to each of  the spaces $\widetilde{P}_C^{(1)} \times \{\frac{1}{2}\}$ in $X$. The preimage of the union of these spaces is a 1--submanifold of $D^2$; that is, a collection of embedded circles and embedded arcs meeting $\partial D^2 = S^1$ transversely in the middle of the edges with labels $e$ or $\bar{e}$. These arcs provide the desired matching of $e$ and $\bar{e}$ edges of $S^1$. 

We use the disk $D^2$ with boundary $S^1$ subdivided into $|\gamma|$ 1--cells with labels pulled back from $X$ together with the collection of arcs connecting the midpoints of $e$ and $\bar{e}$ edges in pairs as a template for building a singular disk diagram (van Kampen diagram) for the loop $\gamma$. We construct (or fill in) the singular disk diagram in two steps as follows.

First, each arc connecting an $e$--$\bar{e}$ pair is mapped into a space $\widetilde{P}_C^{(1)} \times \{\frac{1}{2}\}$ in $X$, the endpoints of the arc sent to vertices of this space. Choose a geodesic edgepath $\lambda$ in $\widetilde{P}_C^{(1)} \times \{\frac{1}{2}\}$ connecting these vertices.  The labels on the edges of the geodesic $\lambda$ are from $S_C$. Now connect the $e$--$\bar{e}$ pair by the strip of 2--cells $\lambda \times e$. This forms an $e$--corridor connecting the pair of edges $e$--$\bar{e}$ of $S^1$. 

Second, note that the union of the $e$--corridors divides the disk $D^2$ into a number of complimentary regions. The number of such regions is one more than the number of $e$--corridors in the first step. This is bounded above by $|\gamma|$. Each region $R_i$ has boundary $\partial R_i$ which maps as a combinatorial edgepath loop into one of the lifts of the spaces ${P}_A$ or ${P}_B$ in $X$. Fill this loop $\partial R_i$ with a singular disk diagram over the corresponding presentation for $A$ or $B$. Glue each such singular disk diagram to the union of $\partial D^2$ and the $e$--corridors by identifying its boundary with $\partial R_i$. The result is a singular disk filling $\Delta$ for $\gamma$. 

Now we estimate the area of $\Delta$ as a function of the length $|\gamma| = x$. All estimates and bounds in the following argument will be up to equivalence. First note that the length of each of the $e$--corridors is bounded above by $\Dist_C^G(x)$. Therefore, the length of the boundary of a complimentary region is bounded above by $$
x + x\Dist_C^G(x) \; \leq \; 2x\Dist_C^G(x)\; \preceq \; x\Dist_C^G(2x)
$$
where the last inequality holds by the fact that $\Dist_C^G$ is equivalent to a super additive function. Therefore, the area of this region is bounded above by 
$$ 
\max\{\delta_A(x\Dist_C^G(2x)), \delta_B(x\Dist_C^G(2x))\}.
$$ 
Finally, there are at most $x$ regions in $\Delta$ and therefore the total area of $\Delta$ is bounded above by the sum of the areas of the $e$--corridors and the areas of the complimentary regions 
$$
x\Dist_C^G(x) + x\max\{\delta_A(x\Dist_C^G(2x)), \delta_B(x\Dist_C^G(2x))\}. 
$$
This is bounded above by 
$2x\max\{\delta_A(x\Dist_C^G(2x)), \delta_B(x\Dist_C^G(2x))\}$ and so the proposition follows from the definition of $\preceq$. \end{proof}

\begin{rem}[application]\label{coolcool} Our primary application of Proposition~\ref{polydehn} involves a situation where $C$ is a retract of $A$. In this case $B$ is a retract of 
 $A\ast_CB$ and so $B$ is undistorted in $A\ast_CB$. Therefore, the distortion of $C$ in $A\ast_CB$ is equivalent to the distortion of $C$ in $B$.
 
 In our application we know that the distortion of $C$ in $B$, the Dehn function of $A$, and the Dehn function of $B$ are all bounded above by polynomial functions. Proposition~\ref{polydehn} implies that the Dehn function of $A\ast_CB$ is also bounded above by a polynomial function. 
\end{rem}

\section{The ambient $\CAT(0)$ groups}

The purpose of this section is to build a $\CAT(0)$ structure for the ambient group $G_n \ast_{H_n}G_n$ of Proposition~\ref{prop:embed0}. 
This structure is built in stages; first chaining together a sequence of $\CAT(0)$ free-by-free groups, then performing a factor-diagonal chaining process (defined in subsection~5.3 below), and finally taking a double. 
The first 2 subsections provide some of the basic building blocks which are used in this construction. 

The goal of subsection~5.1 is to construct 
2--dimensional 
$\CAT(0)$ groups which are isomorphic to $F_\ell \rtimes F_k$ for $k \geq 1$ and which play the role of the groups $(A_i \rtimes B_i)$ in Proposition~\ref{prop:embed0}. This is the content of Theorem~\ref{prop:fbyf}. The third condition in Theorem~\ref{prop:fbyf} ensures that the group $F_k$ is ``ultra-convex" (see Definition~\ref{defn:uc} below) in $F_\ell \rtimes F_k$. This allows us to glue the complex $Y_{k_1}$ for $F_{k_2} \rtimes F_{k_1}$ to the complex $Y_{k_2}$ for $F_{k_3} \rtimes F_{k_2}$ by isometrically identifying the roses $R_{k_2}$ in each space and obtain a non-positively curved result.

In subsection 5.2 we give examples of particular free-by-cyclic groups which play the role of the group $H_0$ in Proposition~\ref{prop:embed0} in various applications.

In subsection 5.3 we use a factor-diagonal chaining construction to combine the spaces obtained in subsections~5.1 and~5.2 together with a non-positively curved spaces corresponding to the terminal vertex group $T$ in Figure~\ref{fig:gog0} in order to build non-positively curved spaces corresponding to the groups $G_n$ and the double $G_n \ast_{H_n}G_n$ of Proposition~\ref{prop:embed0}. 


\subsection{Building 2--dimensional, hyperbolic, $\CAT(0)$ $F_\ell \rtimes F_k$ groups.} We start with a definition. 

\begin{defn}\label{defn:uc}
Let $X$ be a non-positively curved, piecewise euclidean 2--complex. A 1--dimensional subcomplex $Y \subseteq X$ is \emph{ultra-convex} if for each 0--cell $v \in Y$ the set of points of $Lk(v,Y)$ are mutually at least $2\pi$ apart in $Lk(v,X)$.

In particular, the free group $\pi_1(Y)$ injects into $\pi_1(X)$. We say that this subgroup is \emph{ultra-convex} in $\pi_1(X)$ 
\end{defn}






Here is the main result of this subsection. 

\begin{thm}[$\CAT(0)$ hyperbolic $(F_\ell \rtimes F_k)$ with ultra-convex $F_k$]\label{prop:fbyf}
Let $k\geq 1$ be an integer. There exists an integer $\ell>k$ and a connected, non-positively curved, 2--dimensional, piecewise euclidean cell complex $Y_k$ whose fundamental group is $(F_\ell \rtimes F_k)$ with the following properties. 
\begin{enumerate}
    \item $Y_k$ has one 0--cell, $v$, and each loop in the link, $Lk(v,Y_k)$, has length strictly greater than $2\pi$. 
    \item The subgroup $F_\ell$, resp.\ $F_k$, is the fundamental group of a rose $R_\ell$, resp.\ $R_k$, in $Y_k^{(1)}$ based at $v$. In the case $k=1$, the rose $R_1$ is a circle and $F_1 = \Z$.
    \item The rose $R_k \subseteq Y_k$ is ultra-convex in $Y_k$. 
\end{enumerate}
\end{thm}

The proof (which occupies the rest of this subsection) follows the general outline of \cite{MR1898153} where one takes a suitable branched cover of the total space of a graph of spaces. However, for the purposes of the present paper, one needs to take special care with the initial choice of the graph of spaces. Our choice in Step 1 below facilitates the proof of the ultra-convexity condition and also ensures that the branched cover can be achieved via a single cyclic covering of the complement of the branch points. 

The proof has a number of technical steps. We provide an overview of the these steps for the reader's reference. The numbers in the list below correspond to the numbers of the steps in the proof.
\begin{enumerate}
    \item  We define $X_k$ as the total space of a graph of 1--dimensional spaces whose underlying graph, $\Theta_k$, is a rose of $k$ circles. We define a rose--valued Morse function $X_k \to \Theta_k$. The space $X_k$ is given a piecewise euclidean 2--complex structure. 
    \item The complex $X_k$ has a single 0--cell, $v$, and we describe the link $Lk(v, X_k)$. 
    \item The space $X_k -\{v\}$ retracts onto a spine which is a graph, $D_k$. This retraction induces a graph immersion from the barycentric subdivision of $Lk(X_k, v)$ to $D_k$. 
    \item List all the short loops in $Lk(v,X_k)$ and express these as nontrivial elements in $H_1(D_k)$.
    \item Construct a finite cyclic covering space of $D_k$ in which each of the nontrivial elements in $H_1(D_k)$ listed above has one preimage. The piecewise euclidean metric on $X_k -\{v\}$ lifts to the corresponding finite, cyclic cover of $X_k -\{v\}$. The completion of this metric yields a branched cover $\widehat{X_k} \to X_k$. There is a $k$--leaved rose $R_k$ which is ultra--convex in $\widehat{X_k}$ and which maps isomorphically to $\Theta_k$ via the Morse function. Also,
    $\pi_1(\widehat{X_k}, v)$ is hyperbolic. 
    \item There is a subcomplex $Y_k \subseteq \widehat{X_k}$ which has $\CAT(0)$, hyperbolic fundamental group of the form $F_\ell \rtimes F_k$. The free-by-free structure is established using the Morse function restricted to $Y_k$. 
\end{enumerate}

\medskip

\noindent
\emph{Step 1. The total space $X_k$ of a graph of spaces.} In this step we describe $X_k$ as the total space of a graph of spaces, give $X_k$ the structure of a piecewise euclidean complex, and define a graph-valued Morse function $X_k \to \Theta_k$. 

\smallskip

\noindent
\emph{The underlying graph.}
Let $k \geq 1$ be an integer and let $\Theta_k$ denote a graph with one $0$--cell, $w$, and directed edges labelled by $x_i$ for $1 \leq i \leq k$.

\smallskip

\noindent
\emph{The graph of spaces.} 
Next we define a graph of spaces based on $\Theta_k$. The vertex space $X_w$ is a bouquet $R_8$ of 8 circles. Denote the 0--cell of $R_8$ by $v$ and the directed edges by $a_i$ for $1\leq i \leq 8$. 
Each edge space $X_{x_j}$ is defined to be a graph $O_8$ consisting of a circuit of 8 bigons as shown in Figure~\ref{fig:octagon}. The edges are oriented as shown and labelled by $\alpha_i$ and $\beta_i$ for $1\leq i \leq 8$. The vertices are labelled $n_i$ and $p_i$ for $1\leq i \leq 4$ with oriented edges directed away from the $n$ (negative) vertices towards the $p$ (positive) vertices. 

\begin{figure}[ht]
    \centering
    \begin{tikzpicture}
    
     \tikzstyle{every node}=[font=\scriptsize]
     
    \filldraw[black]  (0.8284, 2) circle (2pt) node[xshift=0.8em, yshift=0.8em]{$n_1$};
    
    \filldraw[black]  (2, 0.8284) circle (2pt) node[xshift=1.2em, yshift=0.2em]{$p_1$};
    
    \filldraw[black]  (2, -0.8284) circle (2pt) node[xshift=1.2em, yshift=-0.2em]{$n_2$};
    
    \filldraw[black]  (0.8284, -2) circle (2pt) node[xshift=0.8em, yshift=-0.8em]{$p_2$};
    
    \filldraw[black]  (-0.8284, -2) circle (2pt) node[xshift=-0.8em, yshift=-0.8em]{$n_3$};
    
    \filldraw[black]  (-2, -0.8284) circle (2pt) node[xshift=-1.2em, yshift=-0.2em]{$p_3$};
    
    \filldraw[black]  (-2, 0.8284) circle (2pt) node[xshift=-1.2em, yshift=0.2em]{$n_4$};
    
    \filldraw[black]  (-0.8284, 2) circle (2pt) node[xshift=-0.8em, yshift=0.8em]{$p_4$};
    
    \tikzset{middlearrow/.style={
        decoration={markings,
            mark= at position 0.5 with {\arrow[#1]{stealth}} ,
        },
        postaction={decorate}
    }
}

    
    \draw[thick, middlearrow] (0.8284, 2) .. controls (0.8*1.4142, 0.8*1.4142) .. (2, 0.8284) node[midway,below, left]{$\beta_1$}; 
    
     \draw[thick, middlearrow] (0.8284, 2) .. controls (1.2*1.4142, 1.2*1.4142) .. (2, 0.8284) node[midway,above right]{$\alpha_1$};
    
    \draw[thick, middlearrow] (-0.8284, -2) .. controls (-0.8*1.4142, -0.8*1.4142) .. (-2, -0.8284) node[midway,above right]{$\beta_5$}; 
    
     \draw[thick, middlearrow] (-0.8284, -2) .. controls (-1.2*1.4142, -1.2*1.4142) .. (-2, -0.8284) node[midway,below left]{$\alpha_5$};
    
    \draw[thick, middlearrow] (2, -0.8284) .. controls (0.8*1.4142, -0.8*1.4142) .. (0.8284, -2) node[midway,above left]{$\beta_3$}; 
    
     \draw[thick, middlearrow] (2, -0.8284) .. controls (1.2*1.4142, -1.2*1.4142) .. (0.8284, -2) node[midway,below right]{$\alpha_3$}; 
    
    
    \draw[thick, middlearrow] (-2, 0.8284) .. controls (-0.8*1.4142, 0.8*1.4142) .. (-0.8284, 2) node[midway,below right]{$\beta_7$}; 
    
     \draw[thick, middlearrow] (-2, 0.8284) .. controls (-1.2*1.4142, 1.2*1.4142) .. (-0.8284, 2) node[midway,above left]{$\alpha_7$}; 
    
      \draw[thick, middlearrow] (0.8284, 2) .. controls (0.8*0, 0.8*2) .. (-0.8284, 2) node[midway,below]{$\beta_8$}; 
    
     \draw[thick, middlearrow] (0.8284, 2) .. controls (1.2*0, 1.2*2) .. (-0.8284, 2) node[midway,above]{$\alpha_8$};
    
     \draw[thick, middlearrow] (-0.8284, -2) .. controls (-0.8*0, -0.8*2) .. (0.8284, -2) node[midway,above]{$\beta_4$}; 
    
     \draw[thick, middlearrow] (-0.8284, -2) .. controls (-1.2*0, -1.2*2) .. (0.8284, -2) node[midway,below]{$\alpha_4$};
    
    
      \draw[thick, middlearrow] (2, -0.8284) .. controls (0.8*2, 0.8*0) .. (2, 0.8284) node[midway,left]{$\beta_2$}; 
    
     \draw[thick, middlearrow] (2, -0.8284) .. controls (1.2*2, 1.2*0) .. (2, 0.8284) node[midway,right]{$\alpha_2$};
    
    
      \draw[thick, middlearrow] (-2, 0.8284) .. controls (-0.8*2, 0.8*0) .. (-2, -0.8284) node[midway,right]{$\beta_6$}; 
    
     \draw[thick, middlearrow] (-2, 0.8284) .. controls (-1.2*2, 1.2*0) .. (-2, -0.8284) node[midway,left]{$\alpha_6$};
    
    \end{tikzpicture}

    \caption{The edge space $X_{x_j} = O_8$ for each $1 \leq j \leq k$.}
    \label{fig:octagon}
\end{figure}
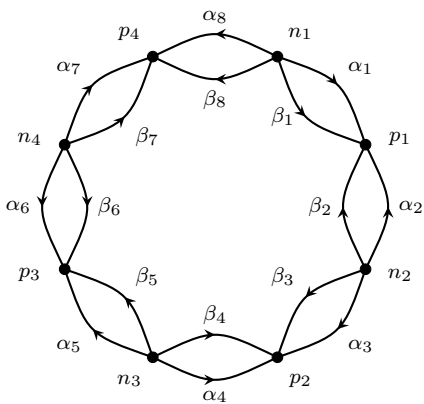

The map $f_1: X_{x_i} \to X_w$ at the terminal end of each edge $x_i$ is given by collapsing the $\beta$--circuit to $v$ and mapping each $\alpha_i$ to $a_i$ for $1 \leq i \leq 8$. 

The map $f_0: X_{x_i} \to X_w$ at the initial end of each edge $x_i$ is given by collapsing the $\alpha$--circuit to $v$ and mapping each $\beta_i$ to $a_{i+4}$ for $1 \leq i \leq 8$ where indices are taken modulo 8 and where we use the digit 8 in place of 0. 

\smallskip

\noindent
\emph{The total space $X_k$.} Recall from (see Scott-Wall~\cite{MR564422}) that the total space $X_k$ of this graph of spaces is defined as the quotient 
$$
X_k \; =\; (X_w \, \sqcup \, \bigsqcup_{1 \leq j \leq k} X_{x_j} \times 
[0,1])/ \sim 
$$
where the equivalence relation is generated by $(x,0) \sim f_0(x)$ and $(x,1) \sim f_1(x)$ for all $x \in X_{x_j}$. We now describe the cell structure of $X_k$. 

Since $X_w = R_8$ has one vertex, $v$, and the underlying graph $\Theta_k$ has one vertex, the complex $X_k$ has only one 0--cell. We denote this 0--cell by $v$. 

Each vertex $n_i$ (resp.\ $p_i$) of $O_8$ determines an oriented edge $n_i \times [0,1]$ (resp.\ $p_i \times [0,1]$) in each $X_{x_j} \times [0,1]$ and hence in $X_k$. We label these edges by $n_{i,j}$ (resp.\ $p_{i,j}$). There are $4k$ of each of these types of edges. The oriented edges $a_i$ for $1 \leq i \leq 8$ of the vertex space $X_w$ also contribute to the 1-skeleton of $X_k$. Therefore, $X_k^{(1)}$ is a bouquet of $8k+8$ oriented edges at the vertex $v$; namely, $p_{i,j}$ and ${n_i,j}$ for $1 \leq i \leq 4$ and $1 \leq j \leq k$ and the edges $a_1, \ldots, a_8$. 

Each edge $\alpha_i$ of $O_8$ determines a 2--cell $(\alpha_i \times [0,1])$ of $X_{x_j} \times [0,1]$. Because $f_0$ collapses $\alpha_i$ to $v \in X_w$, this 2--cell factors through a 2--simplex $(\alpha_i \times [0,1])/ (x,0)\sim (x',0)$ in $X_k$ for each $1 \leq j \leq k$. We denote these 2--simplices by $\alpha_{i,j}$. Likewise, each edge $\beta_i$ of $O_8$ determines $k$ 2--simplices of $X_k$ denoted by $\beta_{i,j}$ for $1 \leq j \leq k$. Figure~\ref{fig:simplices} shows how these simplices glue together in cycles of length 8.

\begin{figure}[ht]
    \centering
    \begin{tikzpicture} 
    
     \tikzstyle{every node}=[font=\scriptsize]
    
    \tikzset{middlearrow/.style={
        decoration={markings,  
            mark= at position 0.5 with {\arrow[#1]{stealth}} ,
        },
        postaction={decorate}
    }
    }
    
    \draw[thick, middlearrow] (0.8284, 2)--(2, 0.8284)  node[midway,above right]{$a_1$};
    
    \draw[thick, middlearrow] (2, -0.8284)--(2, 0.8284)  node[midway,right]{$a_2$};
    
    \draw[thick, middlearrow] (2, -0.8284)--(0.8284, -2)  node[midway,below right]{$a_3$};
    
    \draw[thick, middlearrow] (-0.8284, -2)--(0.8284, -2)  node[midway,below]{$a_4$};
    
    \draw[thick, middlearrow] (-0.8284, -2)--(-2, -0.8284)  node[midway,below left]{$a_5$};
    
    \draw[thick, middlearrow] (-2, 0.8284)--(-2, -0.8284)  node[midway,left]{$a_6$};
    
    \draw[thick, middlearrow] (-2, 0.8284)--(-0.8284, 2)  node[midway,above left]{$a_7$};
    
    \draw[thick, middlearrow] (0.8284, 2)--(-0.8284, 2)  node[midway,above]{$a_8$};

    \draw[thick, middlearrow] (8+0.8284, 2)--(8+2, 0.8284)  node[midway,above right]{$a_5$};
    
    \draw[thick, middlearrow] (8+2, -0.8284)--(8+2, 0.8284)  node[midway,right]{$a_6$};
    
    \draw[thick, middlearrow] (8+2, -0.8284)--(8+0.8284, -2)  node[midway,below right]{$a_7$};
    
    \draw[thick, middlearrow] (8-0.8284, -2)--(8+0.8284, -2)  node[midway,below]{$a_8$};
    
    \draw[thick, middlearrow] (8-0.8284, -2)--(8-2, -0.8284)  node[midway,below left]{$a_1$};
    
    \draw[thick, middlearrow] (8-2, 0.8284)--(8-2, -0.8284)  node[midway,left]{$a_2$};
    
    \draw[thick, middlearrow] (8-2, 0.8284)--(8-0.8284, 2)  node[midway,above left]{$a_3$};
    
    \draw[thick, middlearrow] (8+0.8284, 2)--(8-0.8284, 2)  node[midway,above]{$a_4$};

    
     \draw[thick, middlearrow] (0,0)--(0.8284,2)  node[midway,below, xshift=0.5em]{$n_{1}$};
    
    \draw[thick, middlearrow] (0,0)--(2, 0.8284)  node[midway,below]{$p_{1}$};
    
    \draw[thick, middlearrow] (0,0)--(2, -0.8284)  node[midway,below]{$n_{2}$};
    
    \draw[thick, middlearrow] (0,0)--(0.8284,-2)  node[midway,below right]{$p_{2}$};
    
    \draw[thick, middlearrow] (0,0)--(-0.8284,-2)  node[midway, below left]{$n_{3}$};
    
    \draw[thick, middlearrow] (0,0)--(-2, -0.8284)  node[midway,below]{$p_{3}$};
    
    \draw[thick, middlearrow] (0,0)--(-2, 0.8284)  node[midway,below]{$n_{4}$};
    
    \draw[thick, middlearrow] (0,0)--(-0.8284,2)  node[midway,below, xshift=-0.4em]{$p_{4}$};

    
     \draw[thick, middlearrow] (8+0.8284,2)--(8,0)  node[midway,below, xshift=0.5em]{$n_{1}$};
    
    \draw[thick, middlearrow] 
    (8+2, 0.8284)--(8,0)  node[midway,below]{$p_{1}$};
    
    \draw[thick, middlearrow] 
    (8+2, -0.8284)--(8,0)  node[midway,below]{$n_{2}$};
    
    \draw[thick, middlearrow] (8+0.8284,-2)--(8,0)  node[midway,below right]{$p_{2}$};
    
    \draw[thick, middlearrow] (8-0.8284,-2)--(8,0)  
    node[midway, below left]{$n_{3}$};
    
    \draw[thick, middlearrow] 
    (8-2, -0.8284)--(8,0)  node[midway,below]{$p_{3}$};
    
    \draw[thick, middlearrow] 
    (8-2, 0.8284)--(8,0)  node[midway,below]{$n_{4}$};
    
    \draw[thick, middlearrow] (8-0.8284,2)--(8,0)  node[midway,below, xshift=-0.4em]{$p_{4}$};

    
    \node[] at (0,1.55) {$\alpha_{8}$};
    \node[] at (0,-1.55) {$\alpha_{4}$};
    \node[] at (1.6,0) {$\alpha_{2}$};
    \node[] at (-1.6,0) {$\alpha_{6}$};
    \node[] at (1.15, 1.15) {$\alpha_{1}$};
    \node[] at (1.2, -1.1) {$\alpha_{3}$};
    \node[] at (-1.2, -1.1) {$\alpha_{5}$};
    \node[] at (-1.15, 1.15) {$\alpha_{7}$};
    
    
    \node[] at (8+0,1.55) {$\beta_{8}$};
    \node[] at (8+0,-1.55) {$\beta_{4}$};
    \node[] at (8+1.6,0) {$\beta_{2}$};
    \node[] at (8-1.6,0) {$\beta_{6}$};
    \node[] at (8+1.15, 1.15) {$\beta_{1}$};
    \node[] at (8+1.2, -1.1) {$\beta_{3}$};
    \node[] at (8-1.2, -1.1) {$\beta_{5}$};
    \node[] at (8-1.15, 1.15) {$\beta_{7}$};
    
    \end{tikzpicture}
    \caption{A template for the $16k$ 2--simplices of $X_k$. The actual 2--simplices are labelled by $\alpha_i \to \alpha_{i,j}$ and $\beta_i \to \beta_{i,j}$, and the corresponding 1--simplices are labelled by 
    $n_i \to n_{i,j}$, $p_i \to p_{i,j}$, and $a_i \to a_i$ for $1 \leq j \leq k$. }
    \label{fig:simplices}
\end{figure}
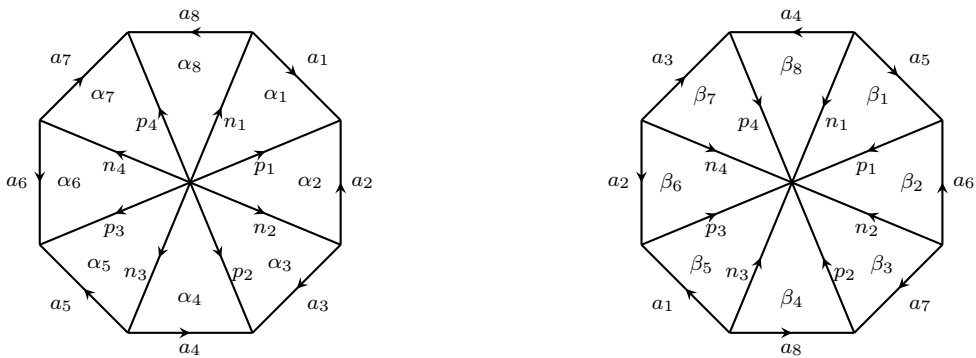

\smallskip

\noindent
\emph{The piecewise euclidean structure of $X_k$.} We give $X_k$ a piecewise euclidean structure in which each 2--cell is a regular euclidean triangle of side length 1.

\smallskip

\noindent
\emph{The Morse function $X_k \to \Theta_k$.} As in \cite{MR1898153} we define a rose-valued Morse function $X_k \to \Theta_k$ which sends the bouquet of 8 edges labelled $a_i$ to $w$, sends each $n_{i,j}$ and each $p_{i,j}$ to $x_j$ for $1 \leq j \leq k$, and which extends linearly over the 2--simplices. In the case $k=1$ this is a standard circle-valued Morse function.

\medskip

\noindent
\emph{Step 2. The link $Lk(X_k, v)$.} One can use the piecewise euclidean metric on $X_k$ and define $Lk(v, X_k)$ to be the boundary of the closed metric ball of radius $1/4$ about $v$ in $X_k$. Each 1-cell $e$ of $X_k$ contributes two vertices $e^+$ and $e^-$ to $Lk(v, X_k)$; the edge $e$ is oriented from $e^-$ towards $e^+$. A pair of adjacent edges in a 2--cell of $X_k$ contributes an edge to $Lk(v,X_k)$.

\begin{figure}[ht]
\centering 
\begin{tikzpicture}

    
     \tikzstyle{every node}=[font=\scriptsize]
    
\begin{knot}[consider self intersections,draft mode=off]
\strand (-1,3.4) -- (1,3.4);
\strand (1,3.4) -- (-1,2.6);
\strand (-1,2.6) -- (1,2.6);
\strand (1,2.6) -- (-1,1.8);
\strand (-1,1.8) -- (1,1.8);
\strand (1,1.8) -- (-1,1);
\strand (-1,1) -- (1,1);
\strand (1,1) -- (-1, 3.4);
\end{knot}

\filldraw[black] (-1,1) circle (2pt) node[xshift=0em, yshift=-0.8em]{$n^+_4$};

\filldraw[black] (-1,1.8) circle (2pt) node[xshift=0em, yshift=-0.8em]{$n^+_3$};

\filldraw[black] (-1,2.6) circle (2pt) node[xshift=0em, yshift=0.8em]{$n^+_2$};

\filldraw[black] (-1,3.4) circle (2pt) node[xshift=0em, yshift=0.8em]{$n^+_1$};

\filldraw[black] (1,1) circle (2pt) node[xshift=0em, yshift=-0.8em]{$p^+_4$};

\filldraw[black] (1,1.8) circle (2pt) node[xshift=0em, yshift=-0.8em]{$p^+_3$};

\filldraw[black] (1,2.6) circle (2pt) node[xshift=0em, yshift=0.8em]{$p^+_2$};

\filldraw[black] (1,3.4) circle (2pt) node[xshift=0em, yshift=0.8em]{$p^+_1$};

\begin{knot}[consider self intersections,draft mode=off]
\strand (-1,-1) -- (1,-1);
\strand (1,-1) -- (-1, -1.8);
\strand (-1,-1.8) -- (1,-1.8);
\strand (1, -1.8) -- (-1, -2.6);
\strand (-1, -2.6) -- (1, -2.6);
\strand (0.99,-2.6) -- (-1,-3.4);
\strand (-1,-3.4) -- (1,-3.4);
\strand (1,-3.4) -- (-1,-1);
\end{knot}

\filldraw[black] (-1,-1) circle (2pt) node[xshift=0em, yshift=0.8em]{$n^-_3$};

\filldraw[black] (-1,-1.8) circle (2pt) node[xshift=0em, yshift=0.8em]{$n^-_4$};

\filldraw[black] (-1,-2.6) circle (2pt) node[xshift=0em, yshift=-0.8em]{$n^-_1$};

\filldraw[black] (-1,-3.4) circle (2pt) node[xshift=0em, yshift=-0.8em]{$n^-_2$};

\filldraw[black] (1,-1) circle (2pt) node[xshift=0em, yshift=0.8em]{$p^-_3$};

\filldraw[black] (1,-1.8) circle (2pt) node[xshift=0em, yshift=0.8em]{$p^-_4$};

\filldraw[black] (1,-2.6) circle (2pt) node[xshift=0em, yshift=-0.8em]{$p^-_1$};

\filldraw[black] (1,-3.4) circle (2pt) node[xshift=0em, yshift=-0.8em]{$p^-_2$};


\filldraw[black] (-5,2.8) circle (2pt) node[xshift=-0.9em, yshift=0.1em]{$a_8^-$};

\filldraw[black] (-5,2.0) circle (2pt) node[xshift=-0.9em, yshift=0.1em]{$a_1^-$};

\filldraw[black] (-5,1.2) circle (2pt) node[xshift=-0.9em, yshift=0.1em]{$a_2^-$};

\filldraw[black] (-5,0.4) circle (2pt) node[xshift=-0.9em, yshift=0.1em]{$a_3^-$};

\filldraw[black] (-5,-0.4) circle (2pt) node[xshift=-0.9em, yshift=0.1em]{$a_4^-$};

\filldraw[black] (-5,-1.2) circle (2pt) node[xshift=-0.9em, yshift=0.1em]{$a_5^-$};

\filldraw[black] (-5,-2.0) circle (2pt) node[xshift=-0.9em, yshift=0.1em]{$a_6^-$};

\filldraw[black] (-5,-2.8) circle (2pt) node[xshift=-0.9em, yshift=0.1em]{$a_7^-$};

\begin{knot}[consider self intersections,draft mode=off]
\strand (-5,2.8) -- (-1,3.4);
\strand (-5,2.8) -- (-1, -1);
\strand (-5,2.0) -- (-1,3.4);
\strand (-5, 2.0) -- (-1, -1);
\strand (-5, 1.2) -- (-1, 2.6);
\strand (-5, 1.2) -- (-1, -1.8);  
\strand (-5,0.4) -- (-1,2.6);
\strand (-5,0.4) -- (-1,-1.8);
\strand (-5,-0.4) -- (-1,1.8);
\strand (-5,-0.4) -- (-1,-2.6);
\strand (-5,-1.2) -- (-1,1.8);
\strand (-5,-1.2) -- (-1,-2.6);
\strand (-5,-2.0) -- (-1,1);
\strand (-5,-2.0) -- (-1,-3.4);
\strand (-5,-2.8) -- (-1,1);
\strand (-5,-2.8) -- (-1,-3.4);
\end{knot}


\filldraw[black] (5,2.8) circle (2pt) node[xshift=0.9em, yshift=0.1em]{$a_1^+$};

\filldraw[black] (5,2.0) circle (2pt) node[xshift=0.9em, yshift=0.1em]{$a_2^+$};

\filldraw[black] (5,1.2) circle (2pt) node[xshift=0.9em, yshift=0.1em]{$a_3^+$};

\filldraw[black] (5,0.4) circle (2pt) node[xshift=0.9em, yshift=0.1em]{$a_4^+$};

\filldraw[black] (5,-0.4) circle (2pt) node[xshift=0.9em, yshift=0.1em]{$a_5^+$};

\filldraw[black] (5,-1.2) circle (2pt) node[xshift=0.9em, yshift=0.1em]{$a_6^+$};

\filldraw[black] (5,-2.0) circle (2pt) node[xshift=0.9em, yshift=0.1em]{$a_7^+$};

\filldraw[black] (5,-2.8) circle (2pt) node[xshift=0.9em, yshift=0.1em]{$a_8^+$};

\begin{knot}[consider self intersections,draft mode=off]
\strand (5,2.8) -- (1,3.4);
\strand (5,2.8) -- (1, -1);
\strand (5,2.0) -- (1,3.4);
\strand (5, 2.0) -- (1, -1);
\strand (5, 1.2) -- (1, 2.6);
\strand (5, 1.2) -- (1, -1.8);  
\strand (5,0.4) -- (1,2.6);
\strand (5,0.4) -- (1,-1.8);
\strand (5,-0.4) -- (1,1.8);
\strand (5,-0.4) -- (1,-2.6);
\strand (5,-1.2) -- (1,1.8);
\strand (5,-1.2) -- (1,-2.6);
\strand (5,-2.0) -- (1,1);
\strand (5,-2.0) -- (1,-3.4);
\strand (5,-2.8) -- (1,1);
\strand (5,-2.8) -- (1,-3.4);
\end{knot}

\end{tikzpicture}
\caption{Template for the link $Lk(v, X_k)$.} 
\label{fig:link}
\end{figure}
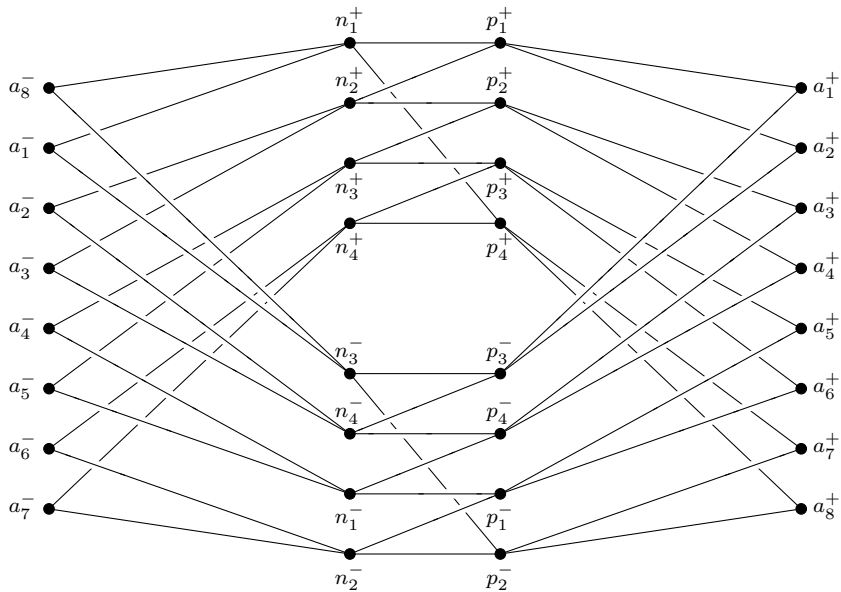

The graph in Figure~\ref{fig:link} is a template for $Lk(v, X_k)$. To obtain the link $Lk(v, X_k)$ first take $k$ copies of the template shown. Next, relabel the vertices in the $j$th copy as follows: $n^\pm_i \to n^\pm_{i,j}$, $p^\pm_i \to p^\pm_{i,j}$ and don't change the labelling of the $a^\pm_i$ vertices. 
Finally, glue these $k$ graphs together along the $a^\pm_i$ vertices. 

The template graph has 32 vertices and 48 edges. Therefore, $Lk(v, X_k)$ obtained from the $k$ copies as above has $16k + 16$ vertices and $48k$ edges. This count agrees with the fact that $X_k$ has $8k + 8$ 1--cells and $16k$ 2--cells since each 1--cell contributes 2 vertices to $LK(v,X_k)$ and each triangular 2--cell contributes 3 edges to $Lk(v, X_k)$. 

Using the terminology and notation of Definition~2.4 and Definition~2.5 in \cite{MR1898153}, the 1--cells $a_i$ (for $1 \leq i \leq 8$) are \emph{horizontal} with respect to the $\Theta_k$--valued Morse function. For each oriented edge $x_j \in \Theta_k$ the \emph{$x_j$--link} of $v$ in $X_k$ is 
given by 
$$
Lk_{x_j}(v, X_k) \; = \; {\hbox{the circle spanned by $n_{i, j}^-$ and $p_{i,j}^-$ for $1 \leq i \leq 4$,}}
$$
and the \emph{$\overline{x}_j$--link} of $v \in X_k$ is 
given by 
$$
Lk_{\overline{x}_j}(v, X_k) \; = \; {\hbox{the circle spanned by $n_{i, j}^+$ and $p_{i,j}^+$ for $1 \leq i \leq 4$.}}
$$
It may help the reader to think of $Lk_{x_j}(v, X_k)$ as the \emph{ascending} (or negative) link of $v$ in the $x_j$--direction and of $Lk_{\overline{x}_j}(v, X_k)$ as the \emph{descending} (or positive) link of $v$ in the $x_j$--direction. 

Note that the links $Lk_{x_j}(a_i, X_k)$ and $Lk_{\overline{x}_j}(a_i, X_k)$ of the horizontal 1--cells $a_i$ are all singletons.

\medskip

\noindent
\emph{Step 3. The spine $D_k$ of $X_k - \{v\}$ and the immersion $Lk(v,X_k)' \to D_k$.} The \emph{spine} of $X_k - \{v\}$ is a graph which we denote by $D_k$ and which is defined as follows. 
The vertices of $D_k$ are the barycenters $\widehat{\alpha}$ of 2-cells $\alpha$ of $X_k$ and the barycenters $\widehat{e}$ of 1--cells $e$ of $X_k$. An $\widehat{\alpha}$ vertex of $D_k$ is adjacent to an $\widehat{e}$ vertex if and only if the 1--cell $e$ is a face of the 2--cell $\alpha$ in $X_k$. 

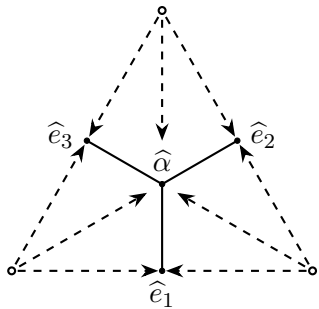
\begin{figure}[ht]
\centering 
\begin{tikzpicture}

\filldraw[black] (0, 1.154) circle (1pt) node[xshift=0em, yshift=0.8em]{$\widehat{\alpha}$};

\filldraw[black] (0, 0) circle (1pt) node[xshift=0em, yshift=-0.8em]{$\widehat{e}_1$};

\filldraw[black] (1, 1.732) circle (1pt) node[xshift=0.9em, yshift=0.2em]{$\widehat{e}_2$};

\filldraw[black] (-1, 1.732) circle (1pt) node[xshift=-0.9em, yshift=0.2em]{$\widehat{e}_3$};

\draw[thick, black] (0, 3.464) circle (1.3pt); 
\draw[thick, black] (-2, 0) circle (1.3pt);
\draw[thick, black] (2, 0) circle (1.3pt);

\draw[->, thick, dashed, -Stealth]  (1.92,0) -- (0.05, 0);
\draw[->, thick, dashed, -Stealth]  (-1.92,0) -- (-0.05, 0);
\draw[->, thick, dashed, -Stealth]  (2- 0.05*1, 0.05*1.732) -- (2- 0.97*1, 0.97*1.732);
\draw[->, thick, dashed, -Stealth]  (-2+ 0.05*1, 0.05*1.732) -- (-2+ 0.97*1, 0.97*1.732);
\draw[->, thick, dashed, -Stealth]  (0.05*1, 3.464-0.05*1.732) -- (0.97*1, 3.464-0.97*1.732);
\draw[->, thick, dashed, -Stealth]  (-0.05*1, 3.464-0.05*1.732) -- (-0.97*1, 3.464-0.97*1.732);

\draw[->, thick, dashed, -Stealth] (-2 + 0.05*2, 0.05*1.154) -- (-2 + 0.9*2, 0.9*1.154);

\draw[->, thick, dashed, -Stealth] (2 - 0.05*2, 0.05*1.154) -- (2 - 0.9*2, 0.9*1.154);

\draw[->, thick, dashed, -Stealth] (0, 3.34) -- (0, 1.5*1.154);

\draw[thick] (0,0) -- (0, 1.154);
\draw[thick] (0, 1.154) -- (1, 1.732); 
\draw[thick] (0, 1.154) -- (-1, 1.732);

\end{tikzpicture}
\caption{The retraction $X_k -\{v\} \to D_k$ inside of a 2--cell $\alpha$ with edges $e_1$, $e_2$, and $e_3$. The dashed lines indicate the retraction; the solid lines are in the spine. }
\label{fig:retract}
\end{figure}

As is shown in Figure~\ref{fig:retract}, the spine graph $D_k$ embeds into $X_k$ and there is a deformation retraction of $X_k -\{v\}$ onto $D_k$.
Each open edge deformation retracts onto its barycenter (midpoint). These deformation retractions extend to a deformation retraction of a 2--cell $
\alpha$ (minus its vertices) onto a tripod graph as indicated in Figure~\ref{fig:retract}. These deformation retractions glue together over open edges to yield a deformation retraction $X_k -\{v\} \to D_k$.

One way to visualize the spine $D_k$ is to embed portions of it into the diagrams in Figure~\ref{fig:simplices}. A template for the spine is described in Figure~\ref{fig:spine}. The actual template is obtained from the graph in Figure~\ref{fig:spine} by identifying vertices in pairs which have the same labels (the $\widehat{a}_i$ vertices). In order to obtain the spine $D_k$ take $k$ copies of the resulting template graph, indexed by $1 \leq j \leq k$, and relabel the vertices in the $j$th copy as follows: $\widehat{\alpha}_i \to \widehat{\alpha}_{i,j}$, 
$\widehat{\beta_i} \to \widehat{\beta}_{i,j}$, and $\widehat{a}_i \to \widehat{a}_i$ for $1 \leq i \leq 8$, and $\widehat{n}_i \to \widehat{n}_{i,j}$, $\widehat{p}_i \to \widehat{p}_{i,j}$ for $1 \leq i \leq 4$. Finally, $D_k$ is obtained by glueing these $k$ copies together by identifying along the vertices $\widehat{a}_i$. 

Each edge in $Lk(v, X_k)$ lies in (the corner of) a 2--cell and is sent by the retraction (see Figure~\ref{fig:retract}) to an edgepath of length 2 in the tripod retract of the 2--cell. In this way the retraction induces a graph map 
$$
\iota: Lk(v, X_k)' \; \to \; D_k
$$
from the first barycentric subdivision of the link to the spine. 
One can see that this map is an immersion of graphs as follows. It is clearly an immersion at the barycenter points. To see that it is an immersion at other vertices, let $e_1$ and $e_2$ be two edges of $Lk(v, X_k)$ based at a vertex of $Lk(v, X_k)$. These edges belong to distinct 2--cells of $X_k$ and so their barycenters are mapped to distinct points in $D_k$. Therefore, the images of these edges are distinct in $D_k$, and so the map is an immersion at all vertices.

Since $\iota: Lk(v, X_k)' \to D_k$ is a graph immersion, it induces an injection $\pi_1(Lk(v, X_k)) \to \pi_1(D_k)$.

\begin{figure}
    \centering
    \begin{tikzpicture}
    
    
     \tikzstyle{every node}=[font=\scriptsize]
    
    \foreach \a in {0,45,...,337.5} 
    { 
    \draw[thick] (\a:3.6cm) -- (\a + 22.5:3.3cm);
    \draw[thick] (\a:3.6cm) -- (\a - 22.5:3.3cm);
    \draw[thick] (\a:2.7cm) -- (\a + 22.5:3.3cm);
    \draw[thick] (\a:2.7cm) -- (\a - 22.5:3.3cm);
    \draw[fill] (\a:3.6cm) circle (1pt);
    \draw[fill] (\a:2.7cm) circle (1pt);
    \draw[fill] (\a+22.5:3.3cm) circle (1pt);
    }

    \foreach \a [count = \b] in {0,45,...,315}
    {
    \draw[thick] (\a:3.6cm) -- (\a: 4cm); 
    \draw[thick] (\a:2.7cm) -- (\a: 1.95cm); 
    \draw[fill] (\a:1.95cm) circle (1pt);
    \draw[fill] (\a:4cm) circle (1pt);
    \draw (\a:4.3cm) node{$\widehat{a}_{\b}$}; 
    \draw (\a:3.3cm) node{$\widehat{\alpha}_{\b}$}; 
    \draw (\a+6:2.5cm) node{$\widehat{\beta}_{\b}$};
    }

    \foreach \a [count = \b] in {22.5, 112.5, 202.5, 292.5}
    {
    \draw (\a:3.6cm) node{$\widehat{p}_{\b}$};
    }
    
    \foreach \a [count = \b] in {-22.5, 67.5, 157.5, 247.5}
    {
    \draw (\a:3.6cm) node{$\widehat{n}_{\b}$};
    }
    
    \draw (0:1.7cm) node{$\widehat{a}_5$};
    \draw (45:1.7cm) node{$\widehat{a}_6$};
    \draw (90:1.7cm) node{$\widehat{a}_7$};
    \draw (135:1.7cm) node{$\widehat{a}_8$};
    \draw (180:1.7cm) node{$\widehat{a}_1$};
    \draw (225:1.7cm) node{$\widehat{a}_2$};
    \draw (270:1.7cm) node{$\widehat{a}_3$};
    \draw (315:1.7cm) node{$\widehat{a}_4$};


    \end{tikzpicture}
    \caption{A template for the spine $D_k$ of $X_k -\{v\}$. The template is obtained from the graph in this figure by identifying vertices with the same labels (the $\widehat{a}_i$ vertices). }
    \label{fig:spine}
\end{figure}
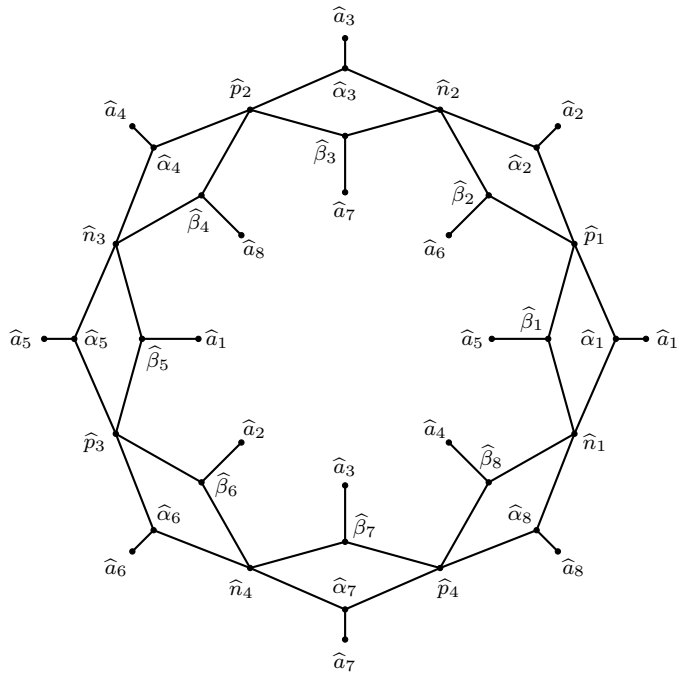

\medskip

\noindent
\emph{Step 4. Computing distance and determining short loops in $Lk(v, X_k)$.} The link $Lk(v, X_k)$ inherits a piecewise spherical structure in which all edges have length $\pi/3$ from the piecewise euclidean structure on $X_k$. In order to obtain a non-positively curved complex, the link needs to be large; that is, all embedded loops in the link have length at least $2\pi$. Combinatorially, this means that all circuits in the link graph have to be of combinatorial length 6 or more. If all loops in the link have combinatorial lengths strictly greater than 6, then the universal cover of $X_k$ is a $\CAT(0)$ space with a cocompact group action by isometries which does not admit any embedded flat planes. By Flat Plane Theorem (see \cite{MR1744486} III.H.1.5) $\pi_1(X_k)$ is hyperbolic. 

The goal is to obtain a branched cover complex of $X_k$, branched over the vertex $v$, with one 0--cell whose link has no loops of combinatorial length 6 or less. This branched cover is obtained by taking a finite covering of $X_k -\{v\}$, lifting the locally euclidean metric to this covering space and taking the metric completion. 

There is an embedded copy of $Lk(v, X_k)$ in $X_k -\{v\}$; namely, the sphere of radius $1/4$ about $v$. In a finite covering space 
$$
\widetilde{X} \; \to \; X_k -\{v\}
$$
the preimage of $Lk(v, X_k)$ is a covering space 
$$
\widetilde{L} \; \to \; Lk(v, X_k).
$$
We want the covering space $\widetilde{X} \to X_k -\{v\}$ and the induced covering space $\widetilde{L} \to Lk(v, X_k)$ to satisfy 4 conditions. 
\begin{enumerate}
    \item The preimage $\widetilde{L}$ is connected. 
    
     This property ensures that the metric completion of the covering space of $X_k -\{v\}$ has just one 0--cell.
     
    \item The preimage of each of the combinatorial length 8 circuits $n^+_{1,j}p^+_{1,j}n^+_{2,j}p^+_{2,j}n^+_{3,j}p^+_{3,j}n^+_{4,j}p^+_{4,j}$ and each of the combinatorial length 8 circuits $n^-_{1,j}p^-_{1,j}n^-_{2,j}p^-_{2,j}n^-_{3,j}p^-_{3,j}n^-_{4,j}p^-_{4,j}$ is connected for $1\leq j \leq k$. 
    
   This significance of this will become apparent later on in the final step of the construction of the free-by-free groups. 
    
    \item The preimage $\widetilde{L}$ does not contain embedded loops of combinatorial length 6 or less. 
    
    This property ensures that the metric completion of the covering space $\widetilde{X}$ of $X_k-\{v\}$ is non-positively curved and that its fundamental group is hyperbolic. 
    
    \item There are lifts of each of the open edges $n_{1,j}$ for $1 \leq j \leq k$ to the covering space of $X_v -\{v\}$ which determine a collection of $2k$ vertices in the covering space $\widetilde{L}$ of $Lk(v, X_k)$ whose mutual combinatorial distances are all at least 6.
    
    This property will be used to establish conclusion (3) of Theorem~\ref{prop:fbyf}. 
\end{enumerate}

Now we describe all loops in $Lk(v, X_k)$ of combinatorial length 6 or less and show that they map to homologically nontrivial 1--cycles in $D_k$. First of all, note that $Lk(v, X_k)$ is a bipartite graph on the disjoint union of vertices 
$$
\{a_1^-, \ldots , a_8^-, p_{i,j}^+, p_{i,j}^- \, |\, 1 \leq i \leq 4,\,  1 \leq j \leq k\}
$$
and 
$$
\{a_1^+, \ldots , a_8^+, n_{i,j}^+, n_{i,j}^- \, |\, 1 \leq i \leq 4,\,  1 \leq j \leq k\}.
$$
This is clear from the template graph in Figure~\ref{fig:link} and the earlier description of how to obtain $Lk(v, X_k)$ from $k$ copies of the template. This means that there are no cycles of odd combinatorial length in $Lk(v, X_k)$; in particular, there are no $m$--cycles for $m=1,3, 5$. 

Now we investigate 2--cycles, 4--cycles, and 6--cycles. By inspection of the template in Figure~\ref{fig:link} there can be no 2--cycles in $Lk(v, X_k)$. 

Next we show that each of the 4--cycles and 6--cycles in $Lk(v, X_k)$ has image in $D_k$ which is homologically nontrivial. To do this, we establish a simple criterion which guarantees that a loop in $Lk(v, X_k)$ has homologically nontrivial image in $D_k$. First, we need some definitions. 

\begin{defn}[type]
Say that a vertex of $D_k$ is of \emph{type 1} (resp.\ \emph{type 2}) if it is the barycenter of a 1--cell (resp.\ 2--cell) of $X_k$.
\end{defn}

Consider a loop $\gamma$ in $Lk(v, X_k)$. By construction/definition, the immersion $\iota: Lk(v, X_k)' \to D_k$ takes vertices of $\gamma$ to type 1 vertices of $D_k$ and barycenters of edges of $\gamma$ to type 2 vertices in $D_k$. 

\begin{defn}[HNT]
A loop $\gamma$ in $Lk(v, X_k)$ satisfies \emph{condition (HNT)} if there is a type 1 vertex in $\iota(\gamma')$ whose preimage in $\gamma$ is a single vertex.
\end{defn}

\begin{lem}
Let $\gamma \subseteq Lk(v, X_k)$ be a loop which satisfies the (HNT) condition. Then $\iota(\gamma')$ is determines a nontrivial element of $H_1(D_k)$.
\end{lem}

\begin{proof}
We argue by contraction. Suppose that $\gamma$ is a 1--cycle in $Lk(v, X_k)$ which satisfies condition (HNT) and $\iota(\gamma')$ is homologically trivial in $D_k$. 

Since $\gamma'$ is a 1--cycle in $Lk(v, X_k)'$ the image $\iota(\gamma')$ is a 1--cycle in $D_k$. By assumption, there is a type 1 vertex $u \in \iota(\gamma')$ whose preimage is just one vertex $u_0 \in \gamma$. Because 
$\iota$ is a graph immersion, the two edges in $\gamma'$ which are adjacent to $u_0$ get sent to distinct edges $e_1$ and $e_2$ of $\iota(\gamma')$ adjacent to $u$. 

Since $\iota(\gamma')$ is homologically trivial in $D_k$, the 1--cells in its image need to be crossed at least twice (with total coefficient sum being zero). In particular, the distinct edges $e_1$ and $e_2$ at $u$ are crossed at least twice. This contradicts the fact that the preimage of $u$ is a single vertex $u_0$. 
\end{proof}

The following terminology will be used in determining the structure of 4--cycles and 6--cycles in $Lk(v, X_k)$. 

\begin{defn}
Edges of $Lk(v,X_k)$ which contain the $a^+_i$ vertices (resp.\ the $a_i^-$ vertices) are called \emph{right edges} (resp.\ \emph{left edges}). The remaining edges (which connect the $n^\pm_{i,j}$ vertices to the $p^{\pm}_{i',j}$ vertices) are called \emph{middle edges}. 
\end{defn}

\begin{lem}\label{lem:4cycle}
The image under the map $\iota: Lk(v,X_k)' \to D_k$ of 
each 4--cycle of $Lk(v, X_k)$ is homologically nontrivial in $D_k$. 
\end{lem}

\begin{proof}
First of all, we argue that there are no 4--cycles which involve the middle edges. Since the middle edges form disjoint circuits of combinatorial length 8, there are no 4--cycles composed entirely of middle edges. A path of 3 consecutive middle edges has one endpoint of the form $n^\pm_{i,j}$ and the other endpoint of the form $p^\pm_{i',j'}$ and so cannot be completed in $Lk(v, X_k)$ to a circuit of combinatorial length 4. If a path of two consecutive middle edges starts at (without loss of generality) some $p^+_{i,j}$ it must end at $p^+_{i', j}$. If we connect one of these endpoints to an $a^+$ vertex by an edge, the other edges from that $a^+$ vertex connect to either $p^-$ vertices or to $p_{a,b}^+$ vertices for $b \not=j$ and so cannot form a 4--cycle. Finally, a circuit which contains a segment of combinatorial length 3 consisting of a single middle edge with a left edge on one end and a right edge on the other has combinatorial length at least 6. 

Therefore, 4--cycles in $Lk(v, X_k)$ are composed entirely of left edges or of right edges. To understand their structure we remove all of the open middle edges from $Lk(v, X_k)$. 
The resulting graph is the disjoint union of 8 joins; 4 composed of right edges and 4 of left edges. Here are two such joins (the first composed of right edges and the second composed of left edges):
$$
\{a_1^+, a_2^+\} \ast \{p_{1,j}^+, p_{3,j}^- \, |\, 1 \leq j \leq k\}   \quad {\hbox{and}} \quad 
\{a_8^-, a_1^-\} \ast \{n_{1,j}^+, n_{3,j}^- \, |\, 1 \leq j \leq k\}. 
$$
The others have a similar structure. In each case, the 4--cycles are obtained by taking the join of the 2 element set with a 2 element subset of the $2k$ element set. This describes all of the nontrivial 4--cycles in $Lk(v, X_k)$.

Consider a nontrivial 4--cycle $\gamma$ in $Lk(v, X_k)$. It is the join of some set $\{a^\pm_i, a^\pm_{i'}\}$ where $i \not=i'$ with another set involving $p^\pm$ or $n^\pm$ vertices. In particular, the images $\widehat{a}_i$ and $\widehat{a}_{i'}$ have singleton preimages in $\gamma'$ and so the image $\iota(\gamma')$ is homologically nontrivial by the (HNT) condition. 
\end{proof}

\begin{lem}\label{lem:6cycle}
The image under the map $\iota: Lk(v,X_k)' \to D_k$ of 
each 6--cycle of $Lk(v, X_k)$ is homologically nontrivial in $D_k$. 
\end{lem}



\begin{proof}
We first observe that each $6$--cycle must contain a middle edge. 
As we observed in Lemma~\ref{lem:4cycle} the complement of the union of the open middle edges in $Lk(v, X_k)$ is the disjoint union of complete bipartite graphs, each of which is a join a 2 vertex set with a $2k$ vertex set. Every immersed segment of combinatorial length 2 which starts and ends in the 2 vertex has distinct endpoints. Therefore, the combinatorial lengths of closed, immersed paths in these bipartite graphs must be $4m$ for positive integers $m$. In particular, there are no 6--cycles.

Let $\gamma$ be an arbitrary $6$--cycle in $Lk(v,X_k)$. By the previous paragraph $\gamma$ contains at least one middle edge, and so must contain a vertex $u$ of the form $p^\pm_{i,j}$. We show that $\{u\}$ is the full preimage of $\{\iota(u)\}$. Assume to the contrary that the preimage of $\iota(u)$ contains two distinct vertices in $\gamma$. Then $\gamma$ must contain the pair of distinct vertices $\{p_{i,j}^+,p_{i,j}^-\}$. Since two vertices $p_{i,j}^+$ and $p_{i,j}^-$ have the same color in the bipartite graph $Lk(v,X_k)$, there is no segment with odd combinatorial length connecting them. Therefore, there is a subsegment of $\gamma$ of combinatorial length $2$ connecting $p_{i,j}^+$ and $p_{i,j}^-$. In particular, $p_{i,j}^+$ and $p_{i,j}^-$ are both adjacent to some vertex in $Lk(v,X_k)$ and we observe that it can not happen. Therefore, the preimage of $\iota(u)$ consists of a single vertex. This implies that $\gamma$ satisfies (HNT) condition and $\iota(\gamma')$ determines a nontrivial element in $H_1(D_k)$.
\end{proof}

The following lemma will be used in establishing the ultra-convexity result (part (3)) of Theorem~\ref{prop:fbyf}. 

\begin{lem}\label{lem:distance}
The combinatorial distance in $Lk(v, X_k)$ between $n_{1,j}^+$ and $n_{1,j}^-$ is at least 6. 
\end{lem}

\begin{proof} 

Note that the map from the link graph to the template graph defined by mapping $a_i^\pm$ to $a_i^\pm$, $n_{i,j}^\pm$ to $n_i^\pm$, $p_{i,j}^\pm$ to $p_i^\pm$, and extending over the 1--skeleton is distance non-increasing. Therefore, it is sufficient to compute distances between $n_1^+$ and $n_1^-$ in the template graph. 
 More precisely, we will show that each immersed segment $\gamma$ connecting $n_1^+$ and $n^-_1$ must have combinatorial length at least $6$. 
 
 Since $n_1^+$ and $n_1^-$ have the same color in the bipartite structure on the template graph, the combinatorial length of $\gamma$ must be even. Therefore, we only need to show the combinatorial length of $\gamma$ is greater than $4$.
We will argue by starting at $n_1^+$ and considering all possibilities for the first edges $e_1$ of $\gamma$. 

There are only four possibilities for $e_1$: one of two left edges or one of two middle edges. If $e_1$ is one of the two left edges adjacent to $n_1^+$, the initial segment of $\gamma$ of combinatorial length $2$ must connect $n_1^+$ to $n_3^-$. Since $n_3^-$ and $n_1^-$ are not both adjacent to some vertex, the distance between them must be greater than $2$. This implies that the combinatorial length $\gamma$ must be greater than $4$. 

We now assume that $e_1$ is a middle edge adjacent to $n_1^+$. Then the other endpoint of $e_1$ is either $p_1^+$ or $p_4^+$. Therefore, the initial segment of combinatorial length $2$ of $\gamma$ must connect $n_1^+$ to one of the vertices in $\{a_1^+, a_2^+, n_2^+, a_7^+, a_8^+, n_4^+\}$. Using the same argument in the previous paragraph we see that the distance between each of these points and $n_1^-$ must be greater than $2$. Therefore, the combinatorial length of $\gamma$ must be greater than $4$.
\end{proof}

\medskip

\noindent
\emph{Step 5. The cyclic branched cover $\widehat{X}_k \to X_k$.} We define a finite, cyclic branched cover $\widehat{X}_k \to X_k$ which is branched over the vertex $v$ by first taking a finite, cyclic covering $\widetilde{X} \to X_k -\{v\}$, lifting the piecewise euclidean metric to $\widetilde{X}$ and then considering the metric completion. Since $D_k$ is a deformation retract of $X_k -\{v\}$ we can model this covering by a cyclic cover of $D_k$. 

The non-positive curvature of $Y_k$ and conclusion (1) of Theorem~\ref{prop:fbyf} will be established by keeping careful control of links as we form the branched cover $\widehat{X}_k$ of $X_k$. In particular, we need to ensure that each of the 4--cycles and 6--cycles in $Lk(v, X_k)$ are unwound nontrivially in the covering space $\widetilde{L}$. Also, we need to ensure that each of the $2k$ loops which are 8--cycles of middle edges has a single preimage in $\widetilde{L}$. 

To this end, define 
${\mathcal L}$ to be the collection of all the nontrivial 4--cycles and 6--cycles and the $2k$ 8--cycles of middle edges. Recall that $\iota: Lk(v, X_k)' \to D_k$ is the graph immersion induced by the retraction map $X_v -\{v\} \to D_k$. Define
$$
{\mathcal C} \; =\; \{ \, [\iota(\gamma')] \in H_1(D_k) \; |\; \gamma \in {\mathcal L}\,\}.
$$
By Lemma~\ref{lem:4cycle} and Lemma~\ref{lem:6cycle} the images of the 4--cycles and 6--cycles are homologically nontrivial in $D_k$. Each of the 8--cycles of middle edges satisfies the (HNT) condition and so has homologically nontrivial image in $D_k$. The set ${\mathcal C}$ is the collection of all of these nontrivial elements of $H_1(D_k)$.

The finite cyclic cover of $D_k$ corresponds to a homomorphism $\pi_1(D_k) \to \Z_N$ for some integer $N>1$. This homomorphism is expressed as the composition
$$
\pi_1(D_k) \; \to \; H_1(D_k) \; \to \; \Z \; \to \; \Z_N
$$
where the middle homomorphism $\ell: H_1(D_k) \to \Z$ is given by inner product with a suitable element of $H_1(D_k)$ chosen so that the image under $\ell$ of each of the elements of ${\mathcal C}$ above is nonzero in $\Z$. See Lemma~3.1 in \cite{MR1898153} for one way of doing this. 

We now describe how to choose $N$, the degree of the cover. 
Recall that $Lk(v, X_k)$ is a subspace of $X_k -\{v\}$ and that the preimage of $Lk(v, X_k)$ in the cyclic cover $\widetilde{X} \to X_k -\{v\}$ is a covering space $\widetilde{L} \to Lk(v, X_k)$. There are two 
things to consider in choosing $1< N \in \Z$. 
\begin{enumerate}
    \item The first consideration concerns non-positive curvature and hyperbolicity. If $N \in \Z$ is chosen to be relatively prime to each of the integers $\ell(c)$ for $c \in {\mathcal C}$, then each loop in ${\mathcal L}$ has preimage a single $N$--fold cover loop in $\widetilde{L}$. 
    
    In particular, $\widetilde{L}$ is connected. This means that the metric completion of the lifted piecewise euclidean metric on $\widetilde{X}$ just adds a single vertex $\hat{v}$. Furthermore, $Lk(\hat{v}, \widehat{X}_v) = \widetilde{L}$ has no loops of combinatorial length less than 8 (because it is bipartite and, by construction, has no loops of combinatorial length 6 or less). Therefore, $\widehat{X}_v$ is a non-positively curved space and $\pi_1(\widehat{X}_v)$ is hyperbolic.

    \item The second consideration concerns conclusion (3) of Theorem~\ref{prop:fbyf}. 
    
    Fix a maximal tree $T_D \subseteq D_k$ in the spine graph. 
    Define ${\mathcal P}_6$ to be the collection of all edgepaths in $Lk(v, X_k)$ of combinatorial length at most 6. Given $\gamma \in {\mathcal P}_6$ the path $\iota(\gamma')$ is an edgepath in $D_k$. There exists a unique geodesic edgepath $\delta_\gamma$ in $T_D$ whose initial point is the terminal point of $\iota(\gamma')$ and whose terminal point is the initial point of $\iota(\gamma')$. The concatenation $\iota(\gamma')\cdot \delta_\gamma$ defines a 1--cycle in $D_k$.
    
     Some of the 1--cycles $\iota(\gamma')\cdot\delta_\gamma$ are homologically nontrivial, and some of these homology classes are sent by the map $\ell: H_1(D_k) \to \Z$ to nontrivial integers. 
    Define an integer $M$ by 
    $$
    M \; =\; \max\left\{\left. |\ell([\iota(\gamma')\cdot\delta_\gamma])|  \; \right|\; \gamma \in {\mathcal P}_6\right\}.
    $$

    One can understand the cyclic covering space $\widetilde{D_k} \to D_k$ by taking the maximal tree 
    $T_D \subset D_k$ chosen above and considering $N$ lifts of it labelled $T^{(1)}_D, \ldots, T^{(N)}_D$ arranged along a line. For each $1 \leq i \leq N$, the vertex $\widehat{n}_{1,j}^{(i)} \in T^{(i)}_D$ is the barycenter (midpoint) of an open 1--cell $n_{1,j}^{(i)} \subset \widetilde{X}$. This 1--cell corresponds to vertices $n_{1,j}^{(i)\pm}$ in the lift $\widetilde{L}$ of the link. 
     
    If we choose $N$ to be larger than $(M+1)k$ and relatively prime to the integers $\ell(c)$ for $c \in {\mathcal C}$ as above (we could take $N$ to be a prime), then the mutual combinatorial distances between the following $2k$ vertices 
    $$
     \left\{ \left. \left(n_{1,j}^{((M+1)j)}\right)^+,\;  \left(n_{1,j}^{((M+1)j)}\right)^- \; \right\rvert \; 1 \leq j \leq k \right\}$$
    in $\widetilde{L}$ are at least 6. Indeed, the combinatorial distance between 
    $\left(n_{1,j}^{((M+1)j)}\right)^+$ and $\left(n_{1,j'}^{((M+1)j')}\right)^-$ for $j\not=j'$ is greater than 6 by our choices of $M$ and $N$ and construction of the cyclic covering. In the case $j=j'$ it is at least 6 because the covering map $\widetilde{L} \to Lk(v, X_k)$ is distance non-increasing and we computed the combinatorial distance between $n_{1,j}^+$ and $n_{1,j}^-$ in $Lk(v, X_k)$ to be at least 6 in Lemma~\ref{lem:distance}. 
    
\end{enumerate}

In conclusion, the metric completion $\widehat{X}_k$ of $\widetilde{X}$ is a piecewise euclidean 2--complex whose 2-cells are euclidean equilateral triangles. It has one 0--cell $\widehat{v}$ and 
$Lk(\widehat{v}, \widehat{X}_k) = \widetilde{L}$ has no loops of combinatorial length less than 8. This means that $\widehat{X}_k$ is non-positively curved and that $\pi_1(\widehat{X}_k)$ is a hyperbolic group.

The complex $\widehat{X}_k$ has $8N$ 1--cells labelled $a_i^{(p)}$ for $1 \leq i \leq 8$ and $1 \leq p \leq N$. It has $4kN$ 1--cells labelled $n_{i,j}^{(p)}$ and $4kN$ 1--cells labelled $p_{i,j}^{(p)}$ for $1\leq p \leq N$, $1\leq i \leq 4$, and $1 \leq j \leq k$. It has $16k N$ triangular 2--cells. 

The branched covering map $\widehat{X}_k \to X_k$ composes with the rose-valued Morse function 
$X_k \to \Theta_k$ to give a rose-valued Morse function $\widehat{X}_k \to \Theta_k$. By construction, this function satisfies the condition $(C_0)$ but not $(HC_1)$ of Definition~2.6 in \cite{MR1898153} and by Corollary~2.9 of that paper the kernel of the induced map 
$\pi_1(\widehat{X}_k, \widehat{v}) \to F_k$ is finitely generated but not finitely presented.
The connectivity of $Lk_{x_j}(\widehat{v}, \widehat{X}_k)$ and of 
$Lk_{\overline{x}_j}(\widehat{v}, \widehat{X}_k)$ follows from the fact that the preimage of each of the $2k$ loops (of combinatorial length 8) composed of middle edges in $Lk(v, X_k)$ is a single circle of combinatorial length $8N$ in $\widetilde{L} = Lk(\widehat{v}, \widehat{X}_x)$. 

The $8N$ 1--cells labelled $a_i^{(p)}$ are horizontal with respect to the Morse function above and the corresponding groups elements, also denoted by $a_i^{(p)}$, generate the kernel of $\pi_1(\widehat{X}_k, \widehat{v}) \to F_k$.

The rose $R_k$ of $k$ edges with labels $n_{1,j}^{((M+1)j)}$ for $1 \leq j \leq k$ is ultra-convex in $\widehat{X}_k$ and determines a free subgroup of $\pi_1(\widehat{X}_k)$ which maps isomorphically to the group $F_k$ under the epimorphism induced by the Morse function.

\medskip

\noindent
\emph{Step 6. The subcomplex $Y_k \subseteq \widehat{X}_k$.} One obtains a subcomplex $Y_k \subseteq \widehat{X}_k$ by removing (any) one of the open horizontal 1--cells $a_i^{(p)}$ together with the interior of each of the $2k$ 2-cells which contain that horizontal $1$--cell. Because it is a subcomplex of $\widehat{X}_k$, the 2--complex $Y_k$ is piecewise euclidean, locally $CAT(0)$ and $\pi_1(Y_k)$ is hyperbolic. Furthermore, the rose $R_k$ is ultra-convex in $Y_k$ and determines a free subgroup $\pi_1(Y_k)$ which maps isomorphically to $F_k$ under the epimorphism induced by the Morse function.

The effect of this operation on the Morse function links is to remove a single edge from each of $Lk_{x_j}(\widehat{v}, \widehat{X}_k)$ and 
$Lk_{\overline{x}_j}(\widehat{v}, \widehat{X}_k)$, turning them from circles of combinatorial length $8N$ into line segments of combinatorial length $8N-1$. Therefore, the Morse function restricted to $Y_k$ satisfies the condition $C_1$ of Definition~2.6 in \cite{MR1898153} and Theorem~2.7(2) of that paper implies that the kernel of the induced homomorphism is free
$$
1 \; \to \; F_{8N-1} \; \to \; \pi_1(Y_k) \; \to \; F_k \; \to \; 1. 
$$
Note that kernel $F_{8N-1}$ is the fundamental group of the rose of the $8N-1$ remaining horizontal 1--cells of $Y_k$. Also, the rose $R_k \subset Y_k$ has fundamental group $F_k$ which maps isomorphically to the quotient $F_k$ in the short exact sequence above. This means that 
$\pi_1(Y_k) \cong (F_{8N-1} \rtimes F_k)$ and the distance in the piecewise spherical metric on $Lk(\widehat{v}, Y_k)$ between any two points of $Lk(\widehat{v}, R_k)$ is at least $6(\pi/3) = 2\pi$. Therefore, the rose $R_k$ and the subgroup $F_k$ are ultra-convex. This concludes the proof of Theorem~\ref{prop:fbyf}.

\subsection{Non-positively curved complexes for $F_p\rtimes \Z$ groups}
\label{freebycyclic} In this subsection we describe various $\CAT(0)$ free-by-cyclic groups which play the role of the group $H_0$ in Proposition~\ref{prop:embed0}. 

There are many examples in the literature of $\CAT(0)$ free-by-cyclic groups whose kernels are exponentially or polynomially distorted of arbitrary degree. 
We give explicit examples of such groups here for completeness and also to show that they can be amalgamated with the free-by-free groups of subsection~5.1. In particular, we show that the geometric model of the $F_k \rtimes \Z$ group contains an embedded rose $R_k$ corresponding to the kernel $F_k$. The kernel $F_k$ plays the role of the group $A_0$ in Proposition~\ref{prop:embed0}. For each class of examples, we specify the corresponding vertex group $S$ and record the distortion of $A_0= F_k$ in $S$. 

\medskip

\noindent
\textbf{Example (1).} \emph{Free-by-cyclic groups with polynomially distorted kernels.} Consider the groups $\Gamma_k$ defined by 
$$
\Gamma_k \; =\; \langle a_1, \ldots , a_k, t \, | \, a_1ta_1^{-1} =t, \; {\hbox{ and $a_ia_{i-1}a_i^{-1} = t$ for $2 \leq i \leq k$}} \rangle. 
$$
The corresponding presentation 2--complex for $\Gamma_k$ has one 0--cell, $(k+1)$ 1--cells labelled $a_1, \ldots, a_k, t$, and $k$ square 2--cells. When given a piecewise euclidean structure in which each 2--cell is a euclidean square of side length 1, this complex is non-positively curved (there are no circuits of combinatorial length less than 4 in the link) and so $\Gamma_k$ is a $\CAT(0)$ group.

There is a circle-valued Morse function on the presentation 2--complex for $\Gamma_k$ obtained by mapping each edge homeomorphically to a circle and extending linearly over the 2--cells. This induces a map $\Gamma_k \to \Z$ sending each of the generators $a_1, \ldots, a_k, t$ to a generator of $\Z$. The ascending and descending links of this Morse function are trees (the ascending link is a cone on the discrete set $a_1^-, \ldots, a_k^-$ with cone point $t^-$ and the descending link is a line segment of combinatorial length $k$ with vertices labelled (in order) $t^+, a_1^+, \ldots, a_k^+$) and so $\Gamma_k$ is an $F_k$-by-$\Z$ group. 

A set of generators of the kernel, $\{x_1, \ldots , x_k\}$, is defined by 
$x_1 =ta_1^{-1}$ and $x_i = a_{i-1}a_i^{-1}$ for $2 \leq i \leq k$. These are represented by the diagonals of the square 2--cells of the presentation 2--complex; these form a rose, $R_k$, of $k$ petals based at the vertex of the presentation 2--complex. 

A representative of the monodromy $\varphi$ is given by conjugation by $t$; namely $\varphi(x_i) = tx_it^{-1}$. One verifies that $\varphi(x_1) = x_1$, $\varphi(x_2) = x_1x_2$ and in general that $\varphi(x_i) = x_1x_2 \ldots x_i$ for $3 \leq i \leq k$. One proves by induction that $\varphi^n(x_1)$ has length 1 and that $\varphi^n(x_i)$ has length equivalent to $n^{i-1}$ for $2 \leq i \leq k$. The induction step uses the fact that 
$$\varphi(x_i) \; = \; x_1 \ldots x_i \; =\; \varphi(x_{i-1})x_i
$$ 
and so $\varphi^n(x_i) = \varphi^{n}(x_{i-1}) \varphi^{n-1}(x_{i-1}) \ldots \varphi(x_{i-1})x_i$ and that these are all positive words. A lower bound of $x^k$ for the distortion of the kernel $F_k$ in $G_k$ is given by the element $t^nx_k^nt^{-n} = (\varphi^n(x_k))^n$. An upper bound is established by considering a word $w(x_i, t)$ of length $x$ which represents an element of $F_k$ and using Britton's Lemma to successively remove innermost $t^\pm\ldots t^\mp$ pinch pairs. Therefore, the distortion of the $F_k$ kernel in $G_k$ is equivalent to $x^k$. In the application, the terminal vertex group $S$ in the top graph of groups in Figure~\ref{fig:gog0} is taken to be $F_k \rtimes_\varphi \Z$ and so the distortion of $F_k$ in $S$ is equivalent to $x^k$. 
 
 These presentation 2--complexes are examples of the the building block complex $K_0$ that is used in the next subsection. In this case, $K_0$ has one 0--cell. 

\medskip

\noindent
\textbf{Example (2).} \emph{$\CAT(0)$ hyperbolic free-by-cyclic groups.} There are many ways of producing $\CAT(0)$ hyperbolic free-by-cyclic groups. For example, one can use the previous subsection with $k=1$ to obtain a non-positively curved piecewise euclidean 2--complex $Y_1$ with fundamental group $F_\ell \rtimes \Z$. By construction, the kernel $F_\ell$ is the fundamental group of a rose, $R_\ell$, in $Y_1$. The distortion of $F_\ell$ in $F_\ell \rtimes \Z$ is exponential because the semidirect product is hyperbolic. In the application, the terminal vertex group $S$ in the top graph of groups in Figure~\ref{fig:gog0} is taken to be $F_\ell \rtimes \Z$ and so the distortion of $F_\ell$ in $S$ is equivalent to $\exp(x)$. 

The complex $Y_1$ is another example of the building block complex $K_0$ that is used in the next subsection. In this case, $K_0$ has one 0--cell. 

\medskip

\noindent
\textbf{Example (3).} \emph{The $\CAT(0)$ $F_2 \rtimes \Z$ group of \cite{MR3705143}.} A key ingredient in the construction in \cite{MR3705143} of 6--dimensional $\CAT(0)$ groups which contain finitely presented snowflake subgroups 
is a particular 
$\CAT(0)$ group of the form $F_2 \rtimes_\varphi \Z$ with palindromic monodromy $\varphi$ of exponential growth.

The group $F_2 \rtimes_\varphi \Z$ is the fundamental group of a 
2--dimensional non-positively curved piecewise euclidean cell complex, labelled $Z_0$ in section 4 of \cite{MR3705143}. This complex has 2 vertices. We subdivide this complex by introducing horizontal edges labelled $x$ and $y$ at one of the two vertices (these are drawn and labelled in Figure~2 of \cite{MR3705143}). This subdivision introduces an embedded rose $R_2$ in the 1--skeleton which represents the kernel $F_2$.

In \cite{MR3705143} there are $\CAT(0)$ groups and snowflake subgroups built from the $m$--fold cyclic covers of $Z_0$ corresponding to $F_2 \rtimes_{\varphi^m}\Z$ for each integer $m \geq 1$. 
The preimage of the rose $R_2$ in the $m$--fold cyclic cover consists of $m$ disjoint roses, each isomorphic to $R_2$. 
Although the distortion of $F_2$ in $F_2 \rtimes_{\varphi^m} \Z$ is exponential, in the application the terminal vertex group $S$ in the top graph of groups in Figure~\ref{fig:gog0} is the snowflake group of \cite{MR3705143}. The distortion of $F_2$ in the snowflake group $S$ is $x^\alpha$. 

The $m$--fold cyclic covers of the subdivided $Z_0$ (for $m \geq 1$) are examples of the building block complex $K_0$ used in the next subsection. In this case, $K_0$ may have more than one 0--cell (the $m$--fold covering of the subdivided $Z_0$ complex has $2m$ 0--cells).

\subsection{Non-positively curved complexes for $G_n$ and $G_n \ast_{H_n} G_n$.}

The $\CAT(0)$ structures for $G_n$ and $G_n \ast_{H_n}G_n$ are built from the blocks in subsections~5.1 and 5.2 together with a $\CAT(0)$ space corresponding to the terminal vertex group $T$ in the lower graph of groups. The building process uses the following three constructions in order:
\begin{enumerate}
    \item Ultra-convex chaining;
    \item Factor-diagonal chaining;
    \item Doubling.
\end{enumerate}
The chaining operations in (1) and (2) are iterated adjunctions of spaces. The reader could see \cite{MR1867354} on page 12 and page 13 for the definition of adjunction (where it is called \emph{attaching spaces}) and notation. We describe these two chaining operations in detail.

\medskip

\noindent
\emph{Ultra-convex chaining.} Start with a 2--dimensional, non-positively curved complex, $K_0$, with fundamental group $F_{k_1} \rtimes \Z$ and a subcomplex $R_{k_1} \subseteq K_0$ which is a rose based at a vertex $v_0 \in K_0$ and which represents the $F_{k_1}$ kernel in $\pi_1(K_0, v_0)$. Three different types of such complexes $K_0$ are described in subsection~5.2 above. 

By repeatedly applying Theorem~\ref{prop:fbyf}, one obtains a sequence of of 2--dimensional, non-positively curved complexes 
$$
K_0, Y_{k_1}, Y_{k_2}, \ldots, Y_{k_n}
$$
with fundamental groups $\pi_1(Y_{k_i}) = F_{k_{i+1}} \rtimes F_{k_i}$, together with roses 
$$
R_{k_1} \subseteq K_0, \; R_{k_2} \subseteq Y_{k_1}, \; \ldots , \; R_{k_n} \subseteq Y_{k_{n-1}}
$$
and locally isometric embeddings 
$f_i: R_{k_i} \hookrightarrow Y_{k_i}$ with ultra-convex images, for $1 \leq i \leq n$. Here, $\pi_1(R_{k_{i+1}})$ is the kernel $F_{k_{i+1}}$ and $\pi_1(f_i(R_{k_i}))$ is the quotient $F_{k_i}$ of $\pi_1(Y_{k_i})$. 

Define iterated adjunction spaces, $K_n$, inductively as follows 
$$
K_i \; =\; Y_{k_i} \sqcup_{f_i} K_{i-1} \quad {\hbox{ for $1 \leq i \leq n$.}} 
$$

The next lemma describes the geometry of $K_n$ and relates its fundamental group to groups $H_i$, defined as in Proposition~\ref{prop:embed0}.

\begin{lem}[ultra-convex chaining]
\label{nested}
Let $K_0, Y_{k_1}, \ldots, Y_{k_n}$ and $K_1, \ldots , K_n$ be as defined above. Define the groups $H_i$ for $0 \leq i \leq n$ by 
$H_0 = \pi_1(K_0)$ and 
$$H_i \; =\;  (F_{k_{i+1}} \rtimes F_{k_i}) \ast_{F_{k_i}} H_{i-1} \text{ for $1\leq i \leq n$}$$ as in Proposition~\ref{prop:embed0}. 
Then the following hold for $1 \leq i \leq n$
\begin{enumerate}
    \item $H_i=\pi_1(K_i)$;
    \item Each $K_i$ is a 2--dimensional, non-positively curved complex; 
    \item Each inclusion $K_{i-1} \hookrightarrow K_{i}$ is a locally isometric embedding.
\end{enumerate}
Moreover, if all loops in the vertex links of $K_0$ are strictly greater than $2\pi$ (so that $H_0$ is a hyperbolic group), then each group $H_i$ is hyperbolic for $i \geq 1$. 
\end{lem}

\begin{proof}
Property (1) follows by induction and van Kampen's theorem. 

By hypothesis $K_0$ is a 2--complex, and each $Y_{k_i}$ is a 2--complex by Theorem~\ref{prop:fbyf}. Therefore, the $K_i$ are 2--complexes.

The complexes $K_i$ are shown to be non-positively curved by induction. Since the $K_i$ are $2$--dimensional piecewise euclidean complexes, it is sufficient to verify that every loop in the link of each vertex of $K_i$ has length at least $2\pi$. This is true for $K_0$ by hypothesis. For each $i\geq 1$ the link of the base vertex $v_0 \in K_i$ is obtained by gluing together the links of base vertices in $K_{i-1}$ and $Y_{k_i}$, along a set of $2k_i$ points (namely, the link of the base vertex in the rose $R_{k_i}$) which is $2\pi$--separated in the latter link
$$
Lk(v_0, X_k) \; =\; Lk(v_0, X_{k-1}) \, \cup_{Lk(v_0, R_{k_i})} Lk(v_0, Y_{k_i}). 
$$
The links of other vertices in $K_{i-1}$ (if any) are unaffected by the adjunction of $Y_{k_i}$. 

 By induction, all loops in the the link of each vertex in $K_{i-1}$ have length at least $2\pi$. By Theorem~\ref{prop:fbyf} all loops in $Lk(v_0, Y_{k_i})$ have length strictly greater than $2\pi$. So we only need to consider the loops created when the link of the base vertex $X_{i-1}$ is glued to the link of the vertex in $Y_{k_1}$ along the set $Lk(v_0, R_{k_i})$. By ultra-convexity, the mutual distances between these $2k_i$ points on the $Y_{k_i}$ side are at least $2\pi$. The distances between these points on the $K_{i-1}$ side are all non-zero. Therefore, the new circuits created in the glueing all have length strictly greater than $2\pi$ and so each $K_i$ is non-positively curved (see Theorem~\ref{lkcondition}). 
 
 Next, we show that each inclusion $K_{i-1} \hookrightarrow K_{i}$ is a locally isometric embedding. To do this, we need to verify that local geodesic paths in $K_{i-1}$ remain local geodesics in $K_{i}$. It suffices to check that this is the case at the base vertex $v_0 \in K_{i-1}$. A local geodesic through $v_0$ determines a pair of points in the link $Lk(v_0, K_{i-1})$ which are at least $\pi$ apart. The only way for this to fail to be a local geodesic in $K_{i}$ is if these two points are less than $\pi$ apart in the link $Lk(v_0, K_{i})$. But this is impossible by the structure of $Lk(v_0, X_k)$ described above and since paths in $Lk(v_0, Y_{k_i})$ which connect points of $Lk(v_0, K_{i})$
are at least $2\pi$ in length.
 
Finally, suppose that all loops in the the link of each vertex in $K_0$ have length strictly greater than $2\pi$. By an analogous induction argument as above, all loops in the the link of each vertex in $K_{i}$ have length strictly greater than $2\pi$. Therefore, the universal cover of each complex $K_i$ is a $\CAT(0)$ space with a cocompact group action by isometries which does not admit any embedded flat planes. By Flat Plane Theorem (see \cite{MR1744486} III.H.1.5) each group $H_i=\pi_1(K_i)$ is hyperbolic.
\end{proof}

\begin{rem}
 The idea of glueing two non-positively curved 2--complexes by identifying an ultra-convex graph in one complex with a (possibly non-convex) copy of the graph in the other complex to obtain a non-positively curved result was used in \cite{MR2252898}. 
\end{rem}

\medskip

\noindent
\emph{Factor-diagonal chaining.} It may be helpful if the reader refers to the schematic diagram in Figure~\ref{fig:heart} for this discussion. Each right triangle in that figure corresponds to a metric product of non-positively curved spaces whose labels are given on the edges adjacent to the right angle. The hypothenuse corresponds to a diagonally embedded subspace in this product space.

The factor-diagonal chaining construction uses two ingredients. 
\begin{enumerate}
    \item The output of the ultra-convex chaining construction. That is, the 2--dimensional, non-positively curved complex $K_n$ together with the nested sequence of locally isometrically embedded subcomplexes $K_0 \subseteq K_1 \subseteq \cdots \subseteq K_n$ of Lemma~\ref{nested}. 
    \item A non-positively curved space $Z_0$ with fundamental group $T$ and a locally isometric embedding $K_0 \hookrightarrow Z_0$, identifying $\pi_1(K_0)$ with the subgroup $H_0\leq T$ of Proposition~\ref{prop:embed0}.
\end{enumerate}

The idea behind the factor-diagonal chaining construction is that $K_i \times K_{i-1}$ contains a locally convex diagonal copy of $K_{i-1}$ which can be glued to the first factor of $K_{i-1} \times K_{i-2}$. However, as we saw in Lemma~\ref{pd}, the diagonal copy of $K_{i-1}$ is isometric to the scaled space $\sqrt{2}K_{i-1}$ so we need to keep scaling by $\sqrt{2}$ in the chaining process. Here are the details. 

\begin{enumerate}
    \item For each integer $1 \leq i \leq n$, define $Z_i$ to be the space $(\sqrt{2})^{n-i}K_i \times (\sqrt{2})^{n-i}K_{i-1}$ with the product metric. It is non-positively curved. Since $K_{i-1}$ locally isometrically embeds into $K_i$, the space $(\sqrt{2})^{n-i}K_{i-1}\times (\sqrt{2})^{n-i}K_{i-1}$ locally isometrically embeds into $(\sqrt{2})^{n-i}K_i \times (\sqrt{2})^{n-i}K_{i-1}$. Therefore, the diagonal $$\Delta_{(\sqrt{2})^{n-i}K_{i-1}} \; \subseteq \; (\sqrt{2})^{n-i}K_{i-1}\times (\sqrt{2})^{n-i}K_{i-1}$$ locally isometrically embeds into $(\sqrt{2})^{n-i}K_i\times (\sqrt{2})^{n-i}K_{i-1} = Z_i$. In the product metric, the subspace $\Delta_{(\sqrt{2})^{n-i}K_{i-1}}$ is isomorphic to the scaled space $$(\sqrt{2})^{n-i+1}K_{i-1} \; =\;  (\sqrt{2})^{n-(i-1)}K_{i-1}.$$ 
    \item For each integer $1 \leq i \leq n$, there is a locally isometric embedding of $\Delta_{(\sqrt{2})^{n-i}K_{i-1}} \subseteq Z_i$ into $Z_{i-1}$ with image the first factor $(\sqrt{2})^{n-(i-1)}K_{i-1}$. Furthermore, there is a locally isometric embedding of 
    $\Delta_{(\sqrt{2})^{n-1}K_0} \subseteq Z_1$ into $(\sqrt{2})^nZ_0$ with image $(\sqrt{2})^nK_0$. 
    \item Define the space $Z$ to be the iterated adjunction of the spaces $(\sqrt{2})^nZ_0, Z_1,  \ldots, Z_n$ using these subspaces and locally isometric embeddings. The space $Z$ is represented by one half of the diagram in Figure~\ref{fig:heart}. 
\end{enumerate}

\begin{figure}[hb]
   \centering
    \begin{tikzpicture}[scale=0.9]
    
    
     \tikzstyle{every node}=[font=\tiny]

     \draw (0.9,4.5) node[font=\normalsize, blue]{$Z_n$}; 
     \draw (-0.9,4.5) node[font=\normalsize, blue]{$Z_n$}; 
      \draw (2.6,4.2) node[blue, font=\normalsize]{$Z_{n-1}$}; 
     \draw (-2.6,4.2) node[blue, font=\normalsize]{$Z_{n-1}$}; 
      \draw (6.6,-0.7) node[blue, font=\normalsize]{$Z_1$}; 
     \draw (-6.6,-0.7) node[blue, font=\normalsize]{$Z_1$}; 
      \draw (-3.4,-2) node[blue, font=\normalsize]{$(\sqrt{2})^n
      Z_0$}; 
     \draw (3.4,-2) node[blue, font=\normalsize]{$(\sqrt{2})^nZ_0$}; 
     
     \draw (6.7,5) node[blue, font=\small]{$Z_i = (\sqrt{2})^{n-i}(K_i \times K_{i-1})$,}; 
     \draw (6.7,4.5) node[blue, font=\small]{for $1\leq i \leq n$.};

     \draw[thick] (0,0)--(0,5) 
     node[near end, above, sloped]{$K_n$}
     node[near end, below, sloped]{$K_n$}; 
      \draw[thick] (2,5)--(0,5)
      node[midway, above]{$K_{n-1}$}; 
     \draw[thick] (0,0)--(2,5) node[midway, above, sloped]{$\Delta_{K_{n-1}}$}
     node[midway, below, sloped]{$\sqrt{2}K_{n-1}$};
     \draw[thick] (0.2, 4.8)--(0, 4.8);
     \draw[thick] (0.2, 4.8)--(0.2,5);

     \draw[thick] (2,5)--({2 + 2*5/sqrt((2)^2 + (5)^2)}, {5 + -2*2/sqrt((2)^2+(5)^2)})
     node[midway, above, sloped]{$\sqrt{2}K_{n-2}$}; 
     
     \draw[thick] (0,0)--({2 + 2*5/sqrt((2)^2 + (5)^2)}, {5 + -2*2/sqrt((2)^2+(5)^2)})
     node[near end, below, sloped]{$(\sqrt{2})^2K_{n-2}$}
     node[near end, above, sloped]{$\Delta_{\sqrt{2}K_{n-2}}$}; 
     
     \draw[thick] ({2 + 0.2*5/sqrt((2)^2 + (5)^2)}, {5 + -0.2*2/sqrt((2)^2+(5)^2)})--
     ({0.96*(2 + 0.2*5/sqrt((2)^2 + (5)^2))}, {0.96*(5 + -0.2*2/sqrt((2)^2+(5)^2))}); 
     
     \draw[thick] (0.96*2, 0.96*5)--
     ({0.96*(2 + 0.2*5/sqrt((2)^2 + (5)^2))}, {0.96*(5 + -0.2*2/sqrt((2)^2+(5)^2))});

     \draw[thick] (0,0)--(7,0) node[midway, above, sloped]{$\Delta_{(\sqrt{2})^{n-2}K_1}$}
     node[midway, below, sloped]{$(\sqrt{2})^{n-1}K_1$};
     \draw[thick] (0,0)--(7,-1.5) 
     node[near end, below, sloped]{$(\sqrt{2})^nK_0$}
     node[near end, above, sloped]{$\Delta_{(\sqrt{2})^{n-1}K_0}$};
     \draw[thick] (7,0)--(7,-1.5) node[midway, above, sloped]{$(\sqrt{2})^{n-1}K_0$};
     
     \draw[thick] (6.8,-0.2)--(7,-0.2);
     \draw[thick] (6.8,-0.2)--(6.8,0);
     
     
     \draw[thick, dashed] ({2 + 2*5/sqrt((2)^2 + (5)^2)}, {5 + -2*2/sqrt((2)^2+(5)^2)}) .. controls (5.6,2.8) .. (7,0);
     
     
     \draw[thick] (0,0)--(1,-2);
     \draw[thick] ({8-140/(49 + 9/4)}, {-3.5 + 30/(49 + 9/4)})--(1,-2); 
     \draw[thick] ({8-140/(49 + 9/4)}, {-3.5 + 30/(49 + 9/4)})--(7,-1.5);

     \draw[thick] (0,0)--(0,5); 
      \draw[thick] (-2,5)--(0,5)
      node[midway, above]{$K_{n-1}$}; 
     \draw[thick] (0,0)--(-2,5) node[midway, below, sloped]{$\sqrt{2}K_{n-1}$}
     node[midway, above, sloped]{$\Delta_{K_{n-1}}$};
     
     \draw[thick] (-0.2, 4.8)--(0, 4.8);
     \draw[thick] (-0.2, 4.8)--(-0.2,5); 

     \draw[thick] (-2,5)--({-(2 + 2*5/sqrt((2)^2 + (5)^2))}, {5 + -2*2/sqrt((2)^2+(5)^2)}) 
     node[midway, above, sloped]{$\sqrt{2}K_{n-2}$}; 
     
     \draw[thick] (0,0)--({-(2 + 2*5/sqrt((2)^2 + (5)^2))}, {5 + -2*2/sqrt((2)^2+(5)^2)}) 
     node[near end, below, sloped]{$(\sqrt{2})^2K_{n-2}$}
     node[near end, above, sloped]{$\Delta_{\sqrt{2}K_{n-2}}$};

       \draw[thick] ({-(2 + 0.2*5/sqrt((2)^2 + (5)^2))}, {5 + -0.2*2/sqrt((2)^2+(5)^2)})--
     ({-0.96*(2 + 0.2*5/sqrt((2)^2 + (5)^2))}, {0.96*(5 + -0.2*2/sqrt((2)^2+(5)^2))}); 
     
     \draw[thick] (-0.96*2, 0.96*5)--
     ({-0.96*(2 + 0.2*5/sqrt((2)^2 + (5)^2))}, {0.96*(5 + -0.2*2/sqrt((2)^2+(5)^2))});


     \draw[thick] (-7,0)--(0,0) 
     node[midway, above, sloped]{$\Delta_{(\sqrt{2})^{n-2}K_1}$}
     node[midway, below, sloped]{$(\sqrt{2})^{n-1}K_1$};
     \draw[thick] (0,0)--(-7,-1.5)
     node[near end, below, sloped]{$(\sqrt{2})^nK_0$}
     node[near end, above, sloped]{$\Delta_{(\sqrt{2})^{n-1}K_0}$};
     \draw[thick] (-7,-1.5)--(-7,0) node[midway, above, sloped]{$(\sqrt{2})^{n-1}K_0$};
     \draw[thick] (-6.8,-0.2)--(-7,-0.2);
     \draw[thick] (-6.8,-0.2)--(-6.8,0);
     
      \draw[thick, dashed] ({-(2 + 2*5/sqrt((2)^2 + (5)^2))}, {5 + -2*2/sqrt((2)^2+(5)^2)}) .. controls (-5.6,2.8) .. (-7,0);
     
     
     \draw[thick] (0,0)--(-1,-2);
     \draw[thick] ({-(8-140/(49 + 9/4))}, {-3.5 + 30/(49 + 9/4)})--(-1,-2); 
     \draw[thick] ({-(8-140/(49 + 9/4))}, {-3.5 + 30/(49 + 9/4)})--(-7,-1.5); 
     
     \end{tikzpicture} 
\caption{The heart of the $\CAT(0)$ construction. The space $Z$ corresponds to one half of this diagram. } \label{fig:heart}
\end{figure}
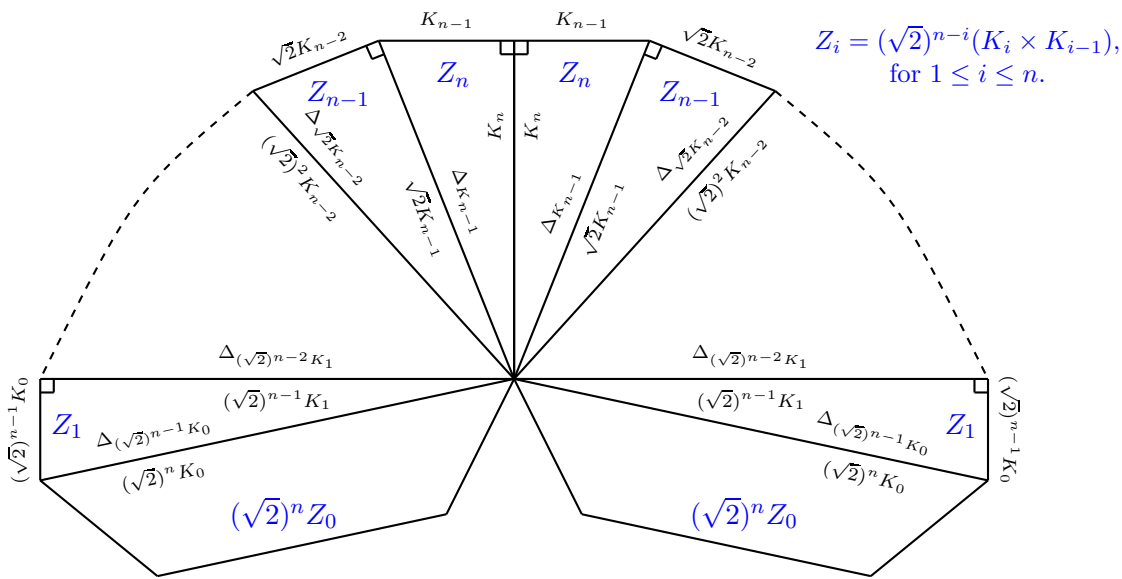

The next lemma describes the geometry of the space $Z$ (resp.\ the double of $Z$ over the subspace $K_n$) and shows that its fundamental group is the group $G_n$ (resp.\ $G_n \ast_{H_n}G_n$) of Proposition~\ref{prop:embed0}.

\begin{lem}[factor-diagonal chaining]
\label{cool}
Let $n$ be a positive integer and let $$K_0 \; \subseteq\;  K_1 \; \subseteq \; \cdots \; \subseteq \; K_n$$ be a sequence of non-positively curved spaces as in Lemma~\ref{nested}. Suppose that $K_0$ locally isometrically embeds into a non-positively curved space $Z_0$ whose fundamental group is $T$, identifying $\pi_1(K_0)$ with the subgroup $H_0 \leq T$ of Proposition~\ref{prop:embed0}. 
\begin{itemize}
    \item Define spaces $Z_i$ for $1 \leq i \leq n$ and $Z$ by the factor-diagonal chaining construction above.
    \item Let $G_1 = (H_1 \times H_0) \ast_{(\Delta_{H_0} \equiv H_0)} T$ and $$G_i \; = \; (H_i \times H_{i-1}) \ast_{(\Delta_{H_{i-1}} \equiv H_{i-1} \times 1)} G_{i-1}\quad {\hbox{for $2 \leq i \leq n$}}$$ be as defined in Proposition~\ref{prop:embed0}. 
\end{itemize}
Then $Z$ is a non-positively curved space 
 satisfying the following properties:
\begin{enumerate}
    \item $Z$ has dimension $\max\{4, \dim Z_0\}$;
    \item $\pi_1(Z)= G_n$; 
    \item  the space $K_n$ locally isometrically embeds as the first factor into $K_n \times K_{n-1} \subseteq Z$ inducing the inclusion of $H_n$ as the first factor $H_n \times 1 \leq H_n \times H_{n-1} \leq G_n$.
\end{enumerate}
Moreover, the double of $Z$ over $K_n$ in (3) above is a non-positively curved space of dimension equal to $\max\{4, \dim Z_0\}$ with fundamental group $G_n\ast_{H_n}G_n$ where 
$H_n$ includes as the first factor in $G_n$. 
\end{lem}

\begin{proof}The space $Z$ is defined by the factor-diagonal chaining above. Since each $K_i$ is 2-dimensional, the products $(\sqrt{2})^{n-i}(K_i \times K_{i-1})$ are 4--dimensional and so $Z$ has dimension equal to $\max\{4, \dim Z_0\}$. 

For each $1 \leq i \leq n$ the inclusion 
$(\sqrt{2})^{n-i+1} K_{i-1} \to (\sqrt{2})^{n-i}(K_i \times K_{i-1})$ is a locally isometric embedding with image $\Delta_{(\sqrt{2})^{n-i}K_{i-1}}$ and induces the diagonal embedding $H_{i-1} \to \Delta_{H_{i-1}} \leq H_i \times H_{i-1}$. Also, for each $1\leq i \leq n$ inclusion $(\sqrt{2})^{n-i}K_i \to (\sqrt{2})^{n-i}(K_i \times K_{i-1})$ is a locally isometric embedding with image the first factor space 
and induces the embedding $H_i \to H_i \times 1 \leq H_i \times H_{i-1}$. Finally, by hypothesis there is a locally isometric embedding $(\sqrt{2})^n K_0 \to (\sqrt{2})^nZ_0$ inducing the group embedding $H_0 \to T$. Using Lemma~\ref{pd}, Proposition~\ref{BH2}, van Kampen's Theorem, and working by induction on $n$ we conclude that $Z$ is a non-positively curved space and has fundamental group $G_n$. 

There is a locally isometric embedding $K_n \to K_n \times K_{n-1} \subseteq Z$ with image the first factor of $K_n \times K_{n-1}$ and which induces the inclusion $H_n \to H_n \times 1 \leq H_n\times H_{n-1}\leq G_n$.

Since the inclusion $K_n \to Z$ above is a locally isometric embedding, the double of $Z$ over $K_n$ is non-positively curved by Lemma~\ref{pd} and Proposition~\ref{BH2}. It has dimension equal to the dimension of $Z$ and has fundamental group $G_n \ast_{H_n}G_n$.
\end{proof}

\medskip

\section{The Main Theorem and open questions}

The main theorem of this paper follows by combining the $\CAT(0)$ constructions of Section~5 for various choices of the space $Z_0$, the group embedding result of Section~3, and the Dehn function computations of Section~4. We recall the construction here. 

\medskip

\noindent
\textbf{Construction.} The construction takes as input the following collection of spaces and groups.
\begin{enumerate}[(i)]
    \item A non-positively curved space $Z_0$ with fundamental group $T$. 
    
    \item A locally isometrically embedded subspace $K_0 \subseteq Z_0$ whose fundamental group is $F_{k_1} \rtimes \Z$ which is identified with a subgroup $H_0 \leq T$. 
    
    \item A rose $R_{k_1} \subseteq K_0$ whose fundamental group is the kernel $F_{k_1} \leq F_{k_1} \rtimes \Z = H_0$. 
    
    \item A subgroup $S \leq T$ such that 
    $$
    S \cap H_0 \; = \; F_{k_1}
    $$
    is the free kernel of $H_0$ and such that 
    $\Dist_{F_{k_1}}^S$ is equivalent to a non-decreasing, super-additive function $f(x)$. 
\end{enumerate}

For $1\leq i\leq n$ let $F_{k_{i+1}}\rtimes F_{k_i}$ be an inductively defined sequence of hyperbolic, $\CAT(0)$ groups which are constructed as in Theorem~\ref{prop:fbyf}. As in Proposition~\ref{prop:embed0}, define sequences of groups $H_i$ and $G_i$ for $1\leq i \leq n$ and $L_i$ for $0 \leq i \leq n$ inductively as follows:
\begin{enumerate}
    \item $H_0$ is the group in (2) above and $$H_i \; =\;  (F_{k_{i+1}} \rtimes F_{k_i}) \ast_{F_{k_i}} H_{i-1} \quad {\hbox{for $1\leq i \leq n$,}}$$ 
    \item $G_1 = (H_1 \times H_0) \ast_{(\Delta_{H_0} \equiv H_0)} T$ where $T$ is given in (1) above and $$G_i \; = \; (H_i \times H_{i-1}) \ast_{(\Delta_{H_{i-1}} \equiv H_{i-1} \times 1)} G_{i-1}\quad {\hbox{for $2 \leq i \leq n$, and}}$$ 
    \item $L_0 = S$ where $S$ is the group given in (4) above and 
    $$L_i \; = \; (F_{k_{i+1}} \rtimes F_{k_i}) \ast_{F_{k_i}} L_{i-1}\quad {\hbox{for $1 \leq i \leq n$.}}$$ 
\end{enumerate}
Let $C_n$ be the double of $G_n$ over $H_n$ and let $D_n$ be the double of $H_n$ over $F_{k_{n+1}}$.

\begin{prop}
\label{main}
$C_n$ is a $\CAT(0)$ group with geometric dimension at most $\max\{4,\dim Z_0\}$ and $C_n$ contains a copy of $D_n$ as a subgroup. Moreover, if the Dehn function of group $L_n$ is dominated by some polynomial function, the the Dehn function of $D_n$ is equivalent to $\exp^{(n)}(f(x))$.
\end{prop}

\begin{proof}
The group $C_n$ is $\CAT(0)$ of dimension at most $\max\{4, \dim Z_0\}$ by Lemma~\ref{cool}. The group embedding $D_n \leq C_n$ is established in Proposition~\ref{prop:embed0}. The Dehn function computation is given in Proposition~\ref{prop:distortion-short}. 
\end{proof}

\begin{thm}[main theorem]
\label{main2}
Let $n$ be a positive integer and define $\exp^{(n)}(x)$ inductively by $\exp^{(1)}(x) = \exp(x)$ and $\exp^{(k+1)}(x) = \exp(\exp^{(k)}(x))$ for $k \geq 1$. Then
\begin{enumerate}
    \item There are $\CAT(0)$ groups of geometric dimension 4 containing finitely presented subgroups whose Dehn functions are $\exp^{(n)}(x^m)$ for integers $m \geq 2$. The maximal rank of free abelian subgroups in these groups is 4.
     \item There are $\CAT(0)$ groups of geometric dimension 4 containing finitely presented subgroups whose Dehn functions are $\exp^{(n)}(x)$ for integers $n \geq 2$. The maximal rank of free abelian subgroups in these groups is 2.
    \item  There are $\CAT(0)$ groups of geometric dimension 6 containing finitely presented subgroups whose Dehn functions are $\exp^{(n)}(x^\alpha)$ for $\alpha$ dense in $[1,\infty)$. The maximal rank of free abelian subgroups in these groups is 6.
\end{enumerate}
\end{thm}

\begin{proof}

We prove the theorem by choosing suitable input data for the construction above.

For Statement (1) we let $H_0 = F_{k_1}\rtimes_{\varphi} \langle t \rangle$ be one of the free-by-cyclic groups from Example~(1) of Subsection~5.2. In that subsection we proved that $\Dist^{F_{k_1}\rtimes_{\varphi} \langle t \rangle}_{F_{k_1}}$ is equivalent to $x^m$ and $F_{k_1}\rtimes_{\varphi} \langle t \rangle$ is the fundamental group of a complex $K_0$ satisfying condition (iii) and half of condition (ii) of the construction above. 
Let $Z_0=K_0\times S_1$ with the product metric. Then $Z_0$ is a non-positively curved space of dimension $3$ and the inclusion $K_0\subset Z_0$ into the first factor is a locally isometric embedding. This gives condition (i) and the remainder of condition (ii). 

Let $T=\pi_1(Z_0)=(F_{k_1}\rtimes_{\varphi} \langle t \rangle) \times \langle u\rangle$ and $H_0=\pi_1(K_0)=F_{k_1}\rtimes_{\varphi} \langle t \rangle$. Let $S$ be the subgroup of $T$ generated by $F_{k_1}$ and $tu$. Then $S$ is the free-by-cyclic group $F_{k_1}\rtimes_{\varphi} \langle tu \rangle$, $S\cap H_0=F_{k_1}$, and the distortion function $\Dist^S_{F_{k_1}}$ is also equivalent to $x^m$. This gives condition (iv) of the construction. 

In this case, the two groups $L_n$ and $H_n$ are isomorphic and therefore they are both $\CAT(0)$ groups by Lemma~\ref{nested}. This implies that the Dehn function of $L_n$ is at most quadratic. Therefore, the $\CAT(0)$ group $C_n$ constructed above contains a subgroup whose Dehn function is $\exp^{(n)}(x^m)$ by Proposition~\ref{main}. 
Also, the group $C_n$ has geometric dimension at most 4 by Proposition~\ref{main} and at least 4 by the existence of $\Z^4$ subgroups. Therefore, the geometric dimension of $C_n$ is exactly $4$ and the maximal rank of free abelian subgroups in $C_n$ is also $4$. 

The proof of Statement (2) is almost identical to the proof of Statement (1) except we let $H_0 = F_{k_1}\rtimes_{\varphi} \langle t \rangle$ be a $\CAT(0)$ hyperbolic free-by-cyclic group as in Example~(2) of Section~\ref{freebycyclic} and we consider the $\CAT(0)$ group $C_{n-1}$ instead of the $\CAT(0)$ group $C_{n}$.
As we proved for Statement (1), the $\CAT(0)$ group $C_{n-1}$ in this case has geometric dimension at most $4$ and it contains a subgroup $D_{n-1}$ whose Dehn function is $\exp^{(n)}(x)$ by Proposition~\ref{main}. Moreover, the subgroup $H_0\times H_0$ of $C_n$ has dimension exactly $4$ (by \cite{MR4072474} and Chapter~VIII Corollary~(7.2) in \cite{MR1324339}). Therefore, the geometric dimension of $C_{n-1}$ is also exactly $4$. 

The group $C_{n-1}$ contains a $\Z^2$ subgroup since (for example) it contains the subgroup $T = (F_{k_1} \rtimes \Z) \times \Z$. We argue that $C_{n-1}$ does not contain $\Z^3$ by contradiction. Assume that $C_{n-1}$ contains $\Z^3$. Note that each group $H_j \leq C_{n-1}$ in this case is hyperbolic (see Lemma~\ref{nested}) and so $C_{n-1} = G_{n-1} \ast_{H_{n-1}} G_{n-1}$ is the fundamental group of a graph of groups whose vertex groups are direct products of two hyperbolic groups and edge groups are hyperbolic. Therefore, the intersection of $\Z^3$ with each edge group conjugate is either $\Z$ or trivial. Since $\Z^3$ does not split nontrivially over $\Z$ or 1, we conclude that $\Z^3$ is conjugate into some vertex group. This is a contradiction since each vertex group is a direct product of two hyperbolic groups and does not contain $\Z^3$. This implies that the maximal rank of free abelian subgroups in $C_{n-1}$ is $2$.



For Statement (3) we let $H_0 = F_2 \times \Z$ be the $\CAT(0)$ group from Example~(3) of Subsection~5.2. In this case, $S$ is the snowflake group and $T$ the ambient $\CAT(0)$ group from \cite{MR3705143}. The precise connection with \cite{MR3705143} is described in Table~\ref{table}.
\begin{table}[ht]
\begin{tabular}{|c | c|}
\hline
{\bf Current paper} & {\bf \cite{MR3705143} paper} \\
\hline
$\CAT(0)$ group $T$ & $\CAT(0)$ group $G_{T,n}$ in Theorem 4.5\\
\hline
$T= \pi_1(Z_0)$ where $Z_0$ is n.p.c. & $G_{T,n} = \pi_1(K_{T,n})$ where $K_{T,n}$ is n.p.c.\\
\hline
Convex subgroup $H_0=F_2\rtimes \Z \leq T$ & Convex subgroup $B_{v_0} \leq G_{T,n}$\\
\hline
$H_0= \pi_1(K_0)$ where  &  
$B_{v_0} = \pi_1(L_0)$ where 
\\
$K_0 \subset Z_0$ is convex & $L_0 \subset K_{T,n}$ is convex (eqn.\ (4.3)) \\
\hline
Subgroup $S$ of $T$ & Subgroup $S_{T,n}$ of $G_{T,n}$ in Theorem 5.6\\

\hline
Subgroup $F_2 = S\cap H_0$ & Subgroup $A_{v_0} = B_{v_0} \cap S_{T,n}$\\
\hline
\end{tabular}
\caption{Correspondence between current construction and \cite{MR3705143}. The notation n.p.c.\ denotes non-positively curved.}\label{table}
\end{table}
Note that the first two rows of the table give condition (i) of the construction. Rows 3 and 4 of the table give condition (ii) of the construction. 
The rose $R_2$ of condition (iii) is described in Example~(3) of Subsection~5.2. Finally, condition (iv) is assured by the last two rows of the table. In particular, the Bass intersection property is a key result in \cite{MR3705143}. 

The distortion $\Dist^{S}_{F_{k_1}}=\Dist^{S}_{F_{2}}$ is computed implicitly in \cite{MR3705143} and is stated explicitly in Lemma 5.6 in \cite{MR4388786}. It is equivalent to $x^\alpha$ for $\alpha$ dense in $[1,\infty)$. Let $L_n$ and $C_n$ be groups constructed above. By construction, $L_n$ is the amalgamation $H\ast_{F_2}S$ where $H$ is a hyperbolic group (see Lemma~\ref{nested}) and $F_2$ is a retract of $H$. 
Moreover, the Dehn function of $S$ is equivalent to $x^{2\alpha}$ (see \cite{MR3705143}) and so is bounded above by a polynomial, the Dehn function of $H$ is linear, and the distortion of $F_2$ in $S$ is equivalent to $x^\alpha$. By Proposition~\ref{polydehn} and Remark~\ref{coolcool} the Dehn function of $L_n$ is dominated by a polynomial. By Proposition~\ref{main}, the $\CAT(0)$ group $C_n$ contains a subgroup whose Dehn function is $\exp^{(n)}(x^\alpha)$. 

We note that the dimension of the space $Z_0$ chosen above in this case is $6$ (see \cite{MR3705143}). Therefore, the geometric dimension of $C_n$ is at most 6 by Proposition~\ref{main} and is at least 6 by the existence of $\Z^6$ subgroups (which are contained in the $T$ vertex groups). Therefore, the geometric dimension of $C_n$ is exactly $6$ and the maximal rank of free abelian subgroups in $C_n$ is also $6$.
\end{proof}

Note that the $\CAT(0)$ ambient groups in the Main Theorem have geometric dimension 4 or 6. Also, the maximum rank of a free abelian subgroup is 2 in the $\exp^n(x)$ examples, 4 in the $\exp^n(x^m)$ (for integers $m \geq 2$) examples, and 6 in the $\exp^{(n)}(x^\alpha)$ examples. Also, these $\CAT(0)$ groups are not relatively hyperbolic with proper peripheral subgroups. These observations (and related considerations) prompt the following questions. 
 
\begin{ques}[3--dimensions] The constructions of the $\CAT(0)$ spaces in the Main Theorem use products and so the ambient $\CAT(0)$ groups have geometric dimension at least 4.
\begin{enumerate}
\item Do there exist $\CAT(0)$ groups of geometric dimension 3 which contain finitely presented subgroups with superexponential Dehn function?
\item Is there an upper bound on the Dehn functions of finitely presented subgroups of $3$--dimensional $\CAT(0)$ groups? 
\end{enumerate}
\end{ques}

Corollary~1.3 of \cite{2020arXiv200811090K} implies that it is impossible to build 2--dimensional examples; indeed, if a finitely presented group coarsely embeds into a $\CAT(0)$ group of geometric dimension 2, then it must satisfy a quadratic isoperimetric inequality. 
 If one is investigating 3--dimensional constructions, it is good to be aware of \cite{MR3276848} which gives a general treatment of tower arguments and their implications for the structure of finitely presented subgroups in various settings.

\begin{ques}[relatively hyperbolic] The prevalence of product structures in our examples mean that the groups are not relatively hyperbolic with proper peripheral subgroups. Indeed, by construction, our groups are geometrically thick and so are not relatively hyperbolic (see Corollary 7.9 in \cite{MR2501302}). The following questions are interesting. 
\begin{enumerate}
    \item Are there $\CAT(0)$ groups which are relatively hyperbolic with abelian peripheral structure and which contain finitely presented subgroups with superexponential Dehn function?
    \item Is there an upper bound on the Dehn functions of finitely presented subgroups of relatively hyperbolic $\CAT(0)$ with abelian peripheral structure?
\end{enumerate}
\end{ques}

The various definitions of relatively hyperbolic groups involve actions on hyperbolic metric spaces. Furthermore, there is a short path from relatively hyperbolic groups to hyperbolic groups via hyperbolic Dehn fillings. Perhaps investigations into the previous set of  questions will shed light on the next set.

\begin{ques}[hyperbolic] The following questions are folklore. 
\begin{enumerate} 
\item Are there hyperbolic groups which contain finitely presented subgroups with Dehn function bounded below by an arbitrary polynomial function? 
\item Are there hyperbolic groups which contain finitely presented subgroups with exponential or superexponential Dehn functions?
\end{enumerate} 
\end{ques}

 There are finitely presented subgroups of hyperbolic groups which are not hyperbolic. See \cite{MR1724853}, \cite{MR3822289}, and \cite{2018arXiv180809505K} for examples which are not of type FP$_3$. In \cite{2021arXiv210514795I} there is an example of a  finite type, non-hyperbolic subgroup of a hyberbolic group. In \cite{2021arXiv211206531L} and \cite{2022arXiv220405788L} there are finitely presented subgroups of hyperbolic groups which are of type FP$_n$ but not FP$_{n+1}$ for $n \geq 3$. The Dehn functions of these subgroups (in all six references above) are bounded below by a quadratic function and are bounded above by a polynomial \cite{MR1934010}.

\begin{ques}[Akermannian] 
Are there $\CAT(0)$ groups containing finitely presented subgroups with Akermannian Dehn functions? 
\end{ques}

The hydra groups of \cite{MR3093501} (and the hyperbolic hydra groups of \cite{MR3134032}) are 2-dimensional $\CAT(0)$ groups containing finite rank free subgroups with Akermannian distortion. One can amalgamate a hyperbolic 
$F_\ell \rtimes F_k$ group with ultra-convex $F_k$ with a hydra group, identifying the $F_k$ with the highly distorted free subgroup in the hydra group. This yields a 2-dimensional $\CAT(0)$ group containing a free-by-free subgroup and a convex hydra subgroup. It is tempting to explore if some variation of the constructions in the current paper can produce subgroups of $\CAT(0)$ groups with Ackermannian Dehn functions. A direct application of the construction in this paper will not work because the highly distorted subgroup of the hydra group is not normal.

\begin{ques}[cubical, RAAG] The construction in the Main Theorem makes use of amalgamations along diagonal subgroups of direct products. This suggests that these $\CAT(0)$ groups might not act properly and cocompactly by isometries on $\CAT(0)$ cube complexes. Here are related questions.
\begin{enumerate}
\item Are there $\CAT(0)$ cubical groups containing finitely presented subgroups with superexponential Dehn functions?
\item Are there RAAGs containing finitely presented subgroups with superexponential Dehn functions?
\end{enumerate}
\end{ques}

It is known that RAAGs contain finitely presented subgroups with exponential Dehn function \cite{MR3151642} and polynomial Dehn function of arbitrary degree \cite{MR3956192}.

\bibliographystyle{alpha}
\bibliography{Tran}
\end{document}